\documentclass[10pt]{amsart}
\usepackage{graphicx}
\usepackage[centertags]{amsmath}
\usepackage{amsfonts}
\usepackage{amssymb}
\usepackage{amsthm}
\newtheorem{thm}{Theorem}

\newtheorem{lem}{Lemma}[section]
\newtheorem{cor}{Corollary}%[section]

\newtheorem{prop}[lem]{Proposition}
\theoremstyle{definition}

\theoremstyle{remark}
\newtheorem{rem}{Remark}[section]
\numberwithin{equation}{section}

\newcommand{\norm}[1]{\left\Vert#1\right\Vert}

\newcommand{\pr}[2]{\langle  #1,#2\rangle}

\newcommand{\G}{\Gamma}

\newcommand{\calA}{\mathcal{A}}

\newcommand{\calF}{\mathcal{F}}
\newcommand{\calG}{\mathcal{G}}
\newcommand{\calH}{\mathcal{H}}

\newcommand{\calO}{\mathcal{O}}

\newcommand{\calR}{\mathcal{R}}

\newcommand{\bbZ}{\mathbb{Z}}

\newcommand{\bbQ}{\mathbb{Q}}
\newcommand{\bbR}{\mathbb R}
\newcommand{\bbC}{\mathbb C}

\newcommand{\bbN}{\mathbb N}

\newcommand{\bbH}{\mathbb{H}}

\newcommand{\Tr}{ \mbox{tr}}

\newcommand{\SL}{ \mathrm{SL}}

\newcommand{\PSL}{ \mathrm{PSL}}

\newcommand{\SU}{ \mathrm{SU}}

\newcommand{\Mat}{\mathrm{Mat}}

\newcommand{\vol}{\mathrm{vol}}

\newcommand{\lap}{\triangle}
\newcommand{\bs}{\backslash}

\newcommand{\frakD}{\mathfrak{D}}

\newcommand{\fraka}{\mathfrak{a}}
\newcommand{\frakp}{\mathfrak{p}}
\newcommand{\frakb}{\mathfrak{b}}
\newcommand{\frakd}{\mathfrak{d}}
\renewcommand{\Im}{\mathrm{Im}}

\begin{document}
\title[Uniform Strong Spectral gaps]
{A Uniform Strong Spectral Gap for Congruence Covers of a compact quotient of  $\PSL(2,\bbR)^d$}%
\author{Dubi Kelmer}%
\address{Department of Mathematics, University of Chicago,  5734 S. University
Ave. Chicago, IL. 60637}
\email{kelmerdu@math.uchicago.edu}

%\thanks{}%

\date{\today}%
\dedicatory{}%
\commby{}%

\begin{abstract}
The existence of a strong spectral gap for lattices in
semi-simple Lie groups is crucial in many
applications. In particular, for arithmetic lattices, it is useful to have bounds for the
strong spectral gap that are uniform in the family of congruence covers.
When the lattice is itself a congruence group, there are uniform and very
good bounds for the spectral gap coming from the known bounds
toward the Ramanujan-Selberg conjectures.
In this note, we establish a uniform bound for the strong spectral gap for congruence covers of an
irreducible co-compact lattice $\Gamma$ in $\PSL(2,\bbR)^d$ with
$d\geq 2$, which is the simplest and most basic case where the
congruence subgroup property is not known.
\end{abstract}

\maketitle
\section*{introduction}
Let $G=\PSL(2,\bbR)$, and let $\Gamma\subset G^d$ denote an arithmetic irreducible co-compact lattice. Note that if $d\geq 2$, then the arithmeticity of $\Gamma$ already follows from Margulis's arithmeticity theorem \cite{Margulis91}. By the classification of arithmetic lattices (see \cite{Weil60}), it follows that any such arithmetic lattice is commensurable to some conjugate of a lattice derived from a quaternion algebra defined over some number field. The space $\Gamma\bs G^d$ has a family of congruence covers $\Gamma(\fraka)\bs G^d$ where $\fraka$ denotes an integral ideal in the corresponding field. In particular, if $\Gamma$ is commensurable to $g^{-1}\Delta g$, we define
$$\Gamma(\fraka)=\Gamma\cap g^{-1}\Delta(\fraka)g,$$
where $\Delta(\fraka)$ denotes the principal congruence group of level $\fraka$. The main result of this paper is a bound for the strong spectral gap that is uniform in the family of congruence covers $\Gamma(\fraka)\bs G^d$.

We refer to \cite{kelmerSarnak09} for a detailed discussion on the strong spectral gap property;
we briefly review the basic definitions and notations.
Given an irreducible unitary representation $\pi$ of $G^d$ on a Hilbert space $\calH$, we denote by $p(\pi)$ the infimum over all $p\geq 2$ such that the matrix coefficients $\pr{\pi(g)v}{v}$ are in $L^p(G)$ for a dense set of vectors $v\in \calH$.
The right regular representation of $G^d$ on the space $L^2(\Gamma\bs G^d)$ decomposes as a direct sum of irreducible representations
\[L^2(\Gamma\bs G^d)=\bigoplus_k m(\pi_k,\Gamma) \calH_{\pi_k},\]
and we define
$$p(\Gamma)=\sup\{ p(\pi_k)| \pi_k \mbox{ non-trivial}\}.$$
As described in \cite{kelmerSarnak09}, $p(\Gamma)$ is finite for any irreducible $\Gamma$ in $G^d$, that is, $\Gamma$ has a strong spectral gap. By a bound for the strong spectral gap we mean a bound for $p(\Gamma)$.

When $\Gamma$ is a congruence group (i.e., $\Gamma\supseteq\Delta(\fraka)$ for some $\fraka$), Selberg's eigenvalue conjecture implies that $p(\Gamma)=2$, and the known bounds toward the Ramanujan-Selberg conjecture
(see \cite{BlomerBrumley10,KimSarnak03}) coupled with the
Jacquet-Langlands correspondence \cite{JacquetLanglands70} yield that $p(\Gamma)\leq \frac{64}{25}$ for all congruence groups.

When $d=1$, there are arithmetic lattices which are not congruence groups. In fact, one can construct lattices with $p(\Gamma)$ arbitrarily large (cf. \cite{Selberg65}).
However, for a family of congruence covers $\Gamma(\fraka)$ with $\Gamma\subset G$ a fixed arithmetic lattice, Sarnak and Xue \cite{SarnakXue91} showed that there is a uniform spectral gap\footnote{While their result actually dealt with prime ideals, it is possible to generalize the method also for composite ideals.}.
In \cite{Gamburd02}, Gamburd used similar ideas to obtain a uniform spectral gap for congruence covers of infinite index subgroups of $\SL(2,\bbZ)$. Such uniform bounds for congruence covers are important in applications as they represent a good substitute for the Selberg-Ramanujan Conjecture when the latter is not available (cf. \cite{BourgainGamburdSarnak06,BourgainGamburdSarnak10,Kontorovich09}).

For $d\geq 2$, Serre conjectures that any irreducible lattice $\Gamma$ in $G^d$ is a congruence one, the case of $\Gamma$ co-compact being the most elementary and fundamental case for which the congruence subgroup problem is open (see \cite[Chapter 7]{LubotzkyAlexander03}). If true, this implies that $p(\Gamma)\leq \frac{64}{25}$ uniformly for all  irreducible lattices in $G^d$. Unconditionally, in joint work with Sarnak \cite{kelmerSarnak09}, we showed that for any fixed $\Gamma\subset G^d$ and $\alpha>0$ there are at most finitely many $\pi$'s occurring in $L^2(\Gamma\bs G^d)$ with $p(\pi)>6+\alpha$. The bound we have for the number of exceptional $\pi$'s depends on $\alpha$ and $\vol(\Gamma\bs G)$ and it goes to infinity when $\alpha\to 0$ or when $\vol(\Gamma\bs G)\to\infty$. In particular, because of the dependence on the volume this result by itself does not give a uniform bound for congruence covers. However, following the suggestion in \cite{kelmerSarnak09}, we combine this with the analysis of \cite{SarnakXue91} to obtain such a uniform bound.

For a congruence cover of level $\fraka$, we say that $\pi$ is a new representation occurring in $L^2(\Gamma(\fraka)\bs G^d)$ if it does not occur in $L^2(\Gamma(\mathfrak{b})\bs G^d)$ for any divisor $\mathfrak{b}|\fraka$.
By strong approximation (see \cite{Weisfeiler84}), there is an ideal $\frakd=\frakd(\Gamma)$, we call the discriminant of $\Gamma$, such that $\Gamma(\fraka)\bs\Gamma\cong \PSL(2,\calO_L/\fraka)$ for all $\fraka$ prime to $\frakd$.  Our main result is then
\begin{thm}\label{t:main}
For any $\delta>0$, there is a constant $N_0=N_0(\Gamma,\delta)$ such that for all ideals $\fraka$ prime to $\mathfrak{d}$ with norm $N(\fraka)>N_0$, there are no new representations occurring in $L^2(\Gamma(\fraka)\bs G^d)$ with $p(\pi)>7+\sqrt{17}+\delta$.
\end{thm}
Since there are only finitely many ideals with bounded norm we get a bound for the strong spectral gap that is uniform for all congruence covers of $\Gamma$, that is,
\begin{cor}\label{c:spectralgap}
For any irreducible co-compact lattice $\Gamma\subseteq G^d$, there is a constant $C(\Gamma)$ such that $p(\Gamma(\fraka))\leq C(\Gamma)$ for all $(\fraka,\frakd)=1$.
\end{cor}

Before we go over the strategy of the proof, let us describe the analysis of \cite{SarnakXue91} in more detail.
A crucial ingredient is the following lower bound for the multiplicities of new representations occurring in $L^2(\Gamma(\fraka)\bs G)$:
For any new representation $\pi$ and ideal $\fraka$ prime to $\frakd$, we have
\begin{equation}\label{e:MLB}
m(\pi,\Gamma(\fraka))\gg_\epsilon N(\fraka)^{1-\epsilon},
\end{equation}
where $N(\fraka)=\#\calO_L/\fraka$. The bound follows from the action of the covering group $\Gamma(\fraka)\bs \Gamma$ on this space and, in particular, from a lower bound on the dimension of irreducible faithful representations of this group.

Another ingredient is the following upper bound for the multiplicities of non-tempered representations in $L^2(\Gamma(\fraka)\bs G)$: Denote by $V(\fraka)=[\Gamma:\Gamma(\fraka)]$ the degree of the cover. By a clever application of the trace formula and a counting argument for lattice points in $\Gamma(\fraka)$, they showed that the multiplicity of any fixed non-tempered representation satisfies
\begin{equation}\label{e:MUB}
m(\pi,\Gamma(\fraka))\ll_\epsilon V(\fraka)^{\frac{2}{p(\pi)}+\epsilon}.
\end{equation}

Combining (\ref{e:MLB}) and (\ref{e:MUB}) and noting that $V(\fraka)\asymp N(\fraka)^3$, we see that if $p(\pi)>6$ and $N(\fraka)$ is sufficiently large, then $\pi$ cannot occur as a new representation in $L^2(\Gamma(\fraka)\bs G)$.

We now return to the case where $d\geq 2$. The multiplicity lower bound (\ref{e:MLB}) for new representations follows from exactly the same arguments as in the case of $d=1$ (see Appendix \ref{s:amult}). Also, by essentially the same argument as in \cite{SarnakXue91}, it is not to hard to show that for a fixed $\pi$, as $V(\fraka)\to\infty$, the multiplicity upper bound (\ref{e:MUB}) also holds. We thus get that if $p(\pi)>6$ and $N(\fraka)$ is sufficiently large, then $\pi$ will not occur as a new representation. However, in contrary to the case of $d=1$, in this case the implied constant in (\ref{e:MUB}), and hence the size of $N(\fraka)$ we require, may depend on the representation $\pi$. Consequently, this argument by itself does not ensure a uniform strong spectral gap for all lattices $\Gamma(\fraka)$.

To make this dependence more precise, we consider the following parameter attached to irreducible representations of $G^d$:
For $\pi\cong \pi_1\otimes\pi_2\otimes\cdots\otimes\pi_d$ an irreducible representation of $G^d$, attach the parameter
\begin{equation}\label{e:Tpi}
T(\pi)=\prod_{j=1}^d(|\lambda(\pi_j)|^{1/2}+1).
\end{equation}
where $\lambda(\pi_j)$ denotes the eigenvalue for the action of the Casimir operator of $G$ on  $\pi_j$.
With this parameter, we can make the dependence explicit, as follows:
\begin{thm}\label{t:multiplicity1}
For $\Gamma\subset G^d$ a fixed irreducible lattice,
\[m(\pi,\Gamma(\fraka))\ll_\epsilon T(\pi)^{\frac{1}{2}+\frac{1}{p(\pi)}+\epsilon}V(\fraka)^{\frac{2}{p(\pi)}+\epsilon}.\]
\end{thm}
As mentioned above, this result by itself does not imply a uniform spectral gap. However, combining this with the lower bound (\ref{e:MLB}), we get a lower bound on $T(\pi)$ for a new representation $\pi$ occurring in $L^2(\Gamma(\fraka)\bs G)$:
\begin{cor}\label{c:lowerbound}
For $\alpha>0$ and an ideal $\fraka$ prime to $\mathfrak{d}$, any new representation $\pi$ occurring in $L^2(\Gamma(\fraka)\bs \Gamma)$ with $p(\pi)>6+\alpha$ satisfies
$T(\pi)\gg_\epsilon V(\fraka)^{\frac{2\alpha}{3(8+\alpha)}-\epsilon}$.
\end{cor}

Hence, if we obtain a sufficiently good upper bound for $T(\pi)$ for these representations, we could get a uniform spectral gap. Such an upper bound can indeed be obtained by refining the analysis in \cite{kelmerSarnak09}.
Specifically, following the arguments in the proof of \cite[Theorem 4]{kelmerSarnak09} adapted to congruence covers, we give an alternative upper bound for the multiplicities.
\begin{thm}\label{t:multiplicity2}
For $\Gamma\subset G^d$ a fixed irreducible lattice and any $c>0$ we have
\[m(\pi,\Gamma(\fraka))\ll_\epsilon V(\fraka)^{1/3}T(\pi)^{2+c(\tfrac{1}{p(\pi)}-\tfrac{1}{2})}\left(V(\fraka)^{2/3}+T(\pi)^{c/2-3+\epsilon}\right).\]
\end{thm}

Combining this upper bound (with a suitable choice of the parameter $c$) and the lower bound (\ref{e:MLB}), we obtain an upper bound for $T(\pi)$. Specifically, we have
\begin{cor}\label{c:upperbound}
For $\alpha>0$ and an ideal $\fraka$ prime to $\mathfrak{d}$, any new representation $\pi$ occurring in $L^2(\Gamma(\fraka)\bs \Gamma)$ with $p(\pi)>6+\alpha$  satisfies
$T(\pi)\ll_\epsilon V(\fraka)^{\frac{4}{3\alpha}+\epsilon}$.
\end{cor}
When $p(\pi)$ is sufficiently large, this upper bound is already smaller than the lower bound given in Corollary \ref{c:lowerbound}, and combining these two bounds gives Theorem \ref{t:main}. Indeed, we have
\begin{proof}[Proof of Theorem \ref{t:main}]
Let $\pi$ denote a new representation occurring in $L^2(\Gamma(\fraka)\bs G^d)$ with $p(\pi)>7+\sqrt{17}+\delta$.
Then, combining Corollaries \ref{c:lowerbound} and \ref{c:upperbound}, we get
\[V(\fraka)^{\frac{2\alpha}{3(8+\alpha)}-\epsilon}\ll_\epsilon T(\pi)\ll_\epsilon V(\fraka)^{\frac{4}{3\alpha}+\epsilon},\]
with $\alpha=1+\sqrt{17}+\delta$. Consequently, there is some constant $C(\epsilon)$ such that
\[V(\fraka)^{\frac{2\alpha^2-4\alpha-32}{3\alpha(8+\alpha)}-\epsilon}\leq C(\epsilon).\]
For this choice of $\alpha$, we have that $\frac{2\alpha^2-4\alpha-32}{3\alpha(8+\alpha)}=c(\delta)>0$ and choosing $\epsilon_0=c(\delta)/2$ sufficiently small, we get that $V(\fraka)$, and hence, $N(\fraka)$ is bounded.
\end{proof}

\begin{rem}
We note that this result is not as good as what was suggested in \cite{kelmerSarnak09}. Even though each of the arguments in \cite{SarnakXue91} and \cite{kelmerSarnak09} separately only break down when $p(\pi)=6$, when combining them we have to assume that $p(\pi)$ is larger than $7+\sqrt{17}\thickapprox 11.12$, in order to make the two bounds on $T(\pi)$ overlap.
\end{rem}

\begin{rem}
The condition that $\fraka$ is prime to $\mathfrak{d}$ may be replaced with the weaker condition that there is a decomposition $\fraka=\fraka_0\fraka_1$ with $N(\fraka_0)$ uniformly bounded and $\fraka_1$ prime to $\frakd$ . In particular, if we restrict to square-free ideals the condition of being prime to $\mathfrak{d}$ can be dropped.
\end{rem}

\begin{rem}
If we consider only spherical representations we can slightly improve Theorem \ref{t:multiplicity1} to show that $m(\pi,\Gamma(\fraka))\ll_\epsilon (T(\pi)V(\fraka))^{\frac{2}{p(\pi)}+\epsilon}$. Consequently, for spherical representations we can replace $7+\sqrt{17}$ in Theorem \ref{t:main} by $6+2\sqrt{2}\thickapprox 8.8$. Since spherical representations correspond to eigenfunctions of the Laplacian $\lap_{z_j}$ in $L^2(\Gamma(\fraka)\bs \bbH\times\cdots\times\bbH)$, we get that when $N(\fraka)$ is sufficiently large any new eigenfunction has eigenvalue bounded below by $\frac{1}{10}$.
\end{rem}

\begin{rem}
An alternative way to define the congruence covers is to first restrict scalars to $GL_{n}(\bbZ)$ and then take elements congruent to $I$ modulo $N$ (for $N\in \bbN$). In our setting, this corresponds to congruence covers $\Gamma(\fraka)$ with $\fraka=N\calO_L$ a principal rational ideal. Since we allow more general ideals, we get, in particular, a uniform spectral gap for these congruence covers.
\end{rem}

\subsection*{Some notation}\label{s:BG}
We write $X\ll Y$ or $X=O(Y)$ to indicate that $X\leq CY$ for some constant $C$.
If we wish to emphasize that the constant depends on some parameters we indicate it by subscripts, for example $X\ll_\epsilon Y$. We will write $X\asymp Y$ to indicate that $X\ll Y\ll X$. We note that all the implied constants may depend on the lattice $\Gamma$ that we think of as fixed.

Throughout this note we denote by $G=\PSL(2,\bbR)$ and $\hat{G}$ its unitary dual,
which we parameterize by $(0,\frac{1}{2})\cup\{\frac{1}{2}+i \bbR^+\}\cup \bbZ$.
In particular, the spherical irreducible representations of $G$ are given by the principal
series representations $\pi_{s},\;s\in\frac{1}{2}+i\bbR^+$ and the complementary series representations
$\pi_s,\;s\in (0,\tfrac{1}{2})$. The non-spherical representations are given by the discrete series $\frakD_{m},\;m\in\bbZ\setminus\{0\}$. With this parametrization, the eigenvalue $\lambda(\pi)$ for the action of the Casimir operator of $G$ on the irreducible representations $\pi$ is given by $\lambda(\pi_s)=s(1-s)$ and $\lambda(\frakD_m)=|m|(1-|m|)$. The discrete and principal series are both tempered while the
complementary series is non-tempered with $p(\pi_s)=\frac{1}{s}$.

An irreducible unitary representation of $G^d$ is of the form $\pi\cong \pi_1\otimes\cdots\otimes\pi_d$ with each $\pi_j\in \hat{G}$. For such a representation $\lambda(\pi)=(\lambda(\pi_1),\ldots,\lambda(\pi_d))\in \bbR^d$ and $p(\pi)=\max_{j}p(\pi_j)$.

We denote by $P,A,K\subset G$ the subgroups of upper triangular matrices, diagonal matrices, and orthogonal matrices respectively. For $t\in \bbR$ and $\theta\in[0,2\pi]$ let $a_t=\left(\begin{smallmatrix}e^{t/2} & 0\\0& e^{-t/2}\end{smallmatrix}\right)\in A$ and $k_\theta=\left(\begin{smallmatrix}\cos(\frac{\theta}{2}) & \sin(\frac{\theta}{2}) \\-\sin(\frac{\theta}{2})& \cos(\frac{\theta}{2})\end{smallmatrix}\right)\in K$. For $z=x+iy\in \bbH$, the upper half plane, let $p_z=\left(\begin{smallmatrix}y^{1/2} & xy^{-1/2} \\ 0 & y^{-1/2}\end{smallmatrix}\right)\in P$.
The decomposition $G=PK$ gives us the coordinate system $g(z,\theta)=p_zk_\theta$. The Haar measure in these coordinates is given by $dg(z,k)=dzdk$ where $dz=\tfrac{dxdy}{y^2}$ and $dk=\frac{d\theta}{2\pi}$. We also have the decomposition $G=KA^+K$, and in the corresponding coordinates, the Haar measure is given by $dg(k,t,k')=2\pi \sinh(t)dtdkdk'$.

For $g=\left(\begin{smallmatrix} a& b\\ c&d\end{smallmatrix}\right)\in \Mat(2,\bbC)$, let $\norm{g}^2=\Tr(g^t\bar{g})=|a|^2+|b|^2+|c|^2+|d|^2$. For $g\in G$, let $H(g)=d(gi,i)$, where $d(\cdot,\cdot)$ denotes the hyperbolic distance on $\bbH$. We thus have that $H(ka_tk')=t$ and a simple calculation shows that $\norm{g}^2=2\cosh(H(g))$ for any $g\in G$.
The triangle inequality for the hyperbolic distance implies that for any $g,\gamma\in G$,
\begin{equation}\label{e:triangle}
|H(g^{-1}\gamma g)-H(\gamma)|\leq 2H(g)
\end{equation}
For any $\delta>0$ we denote by $B_\delta\subset G$ the ball $B_\delta=\{g\in G|H(g)\leq \delta\}$ so that
$|H(g^{-1}\gamma g)-H(\gamma)|\leq 2\delta$ for all $g\in B_\delta$.
%For $g=\left(\begin{smallmatrix} a& b\\ c&d\end{smallmatrix}\right)\in \PSL(2,\bbR)$ we denote by $j_g(z)=\frac{cz+d}{c\bar{z}+d}$.
%Note that if $gp_z=p_{gz}k_\theta$ then $j_g(z)=e^{-i\theta}$.
\subsection*{Acknowledgements}
We thank Peter Sarnak for discussions about this paper.

\section{Counting Lattice Points}
In this section we give some bounds for the number of lattice points in certain regions in $\Gamma(\fraka)$. As we are interested in upper bounds, it is sufficient to consider only the principal congruence groups $\Delta(\fraka)$. We briefly go over the construction of these groups.
\subsection{Lattices derived from quaternion algebras}\label{s:QuatAlg}
Let $L$ be a totally real number field of degree $[L:\bbQ]=n\geq d$; denote by $\calO_L$ its ring of integers and let
$\iota_1,\ldots,\iota_n$ denote the different embeddings of $L$ into
$\bbR$. Let $p,q\in \calO_L$ and assume that $\iota_j(p)>0,\;\iota_j(q)<0$ for $j=1,\ldots,d$ and $\iota_j(p)<0,\;\iota_j(q)<0$ for $j>d$. The quaternion algebra $\calA=\left(\frac{p,q}{L}\right)$ is the $4$ dimensional algebra over $L$ generated by $1,I,J,$ and $K$ with relations
\[I^2=p,J^2=q,K=IJ=-JI.\]
For each $j\leq d$ we have $\calA\otimes_{\iota_j(L)} \bbR\cong \Mat(2,\bbR)$ while for $j>d$ we have that
$\calA\otimes_{\iota_j(L)} \bbR$ is isomorphic to Hamilton Quaternions.

For each $j\leq d$ (respectively $j>d$), the embedding $\iota_j$ induces an embedding (that we still denote by $\iota_j$) of $\calA$ into $\Mat(2,\bbR)$ (respectively $\Mat(2,\bbC)$) sending $\alpha=a+bI+cJ+dK$ to
\begin{equation}\iota_j(\alpha)=\begin{pmatrix}\iota_j(a)+\iota_j(b)\sqrt{\iota_j(p)}& \iota_j(cq)+\iota_j(dq)\sqrt{\iota_j(p)}\\ \iota_j(c)-\iota_j(d)\sqrt{\iota_j(p)}& \iota_j(a)-\iota_j(b)\sqrt{\iota_j(p)}\end{pmatrix}.\end{equation}

For each $\alpha=a+bI+cJ+dK\in \calA$ the relative norm and trace of $\alpha$ are defined as
\begin{eqnarray}
n_\calA(\alpha)&=&a^2-pb^2-qc^2+pqd^2,\\
\Tr_\calA(\alpha)&=&2a.
\end{eqnarray}

Under the above imbedding $\iota_j(n_\calA(\alpha))=\det(\iota_j(\alpha))$ and $\iota_j(\Tr_\calA(\alpha))=\Tr(\iota_j(\alpha))$.
In particular, the group of norm one elements, $\calA^1=\{\alpha\in\calA|n_\calA(\alpha)=1\}$,
satisfies that $\iota_j(\calA^1)\subset\SL(2,\bbR)$ for $j\leq d$ and $\iota_j(\calA^1)\subset\SU(2)$ for $j>d$.

For $\alpha\in \calA^1$ let $\iota(\alpha)=(\iota_1(\alpha),\ldots,\iota_d(\alpha))\in\SL(2,\bbR)^d$.
Let
$$\calR=\{\alpha=a+bI+cJ+dK\in \calA|a,b,c,d\in\calO_L\},$$
and $\calR^1=\calR\cap\calA^1$.
The image $\iota(\calR^1)$ is an irreducible lattice in $\SL(2,\bbR)^d$. This lattice is co-compact unless $\calA\cong \left(\frac{-1,1}{L}\right)\cong\Mat(2,L)$ (this can happen only when $n=d$).
Projecting this lattice modulo $\pm I$ gives a lattice  $\Delta\subset G^d$ that we call the lattice derived from the quaternion algebra $\calA$.

For any ideal $\fraka\in \calO_L$, the group
$$\calR^1(\fraka)=\{\alpha=a+bI+cJ+dK\in \calA|a-1,b,c,d\in\fraka\},$$
is a subgroup of finite index in $\calR^1$, and the principal congruence group $\Delta(\fraka)$ is defined as the image of $\calR^1(\fraka)$ under the projection to $G^d$.

\subsection{Counting functions}
The multiplicity bound (\ref{e:MUB}) for non-tempered representations occurring in $L^2(\Gamma(\fraka)\bs G)$ was obtained in \cite{SarnakXue91} from an estimate of the counting function
$$N(x;\fraka)=\sharp\{ 1\neq\alpha\in \calR^1(\fraka)|\norm{\alpha}_1^2\leq x,\;\forall j>1\;\norm{\alpha}_j^2\leq C\},$$
where we denote by $\norm{\alpha}_j=\norm{\iota_j(\alpha)}$. When $d=1$, the condition that $\norm{\alpha}_j$ is uniformly bounded for $j>1$ is automatic. When $d>2$, this is obviously not the case, but since we impose this condition in the counting function, then \cite[Lemma 3.2]{SarnakXue91} still gives the same bound, which is
\begin{equation}\label{e:counting1}
N(x;\fraka)\ll_\epsilon \frac{x^{1+\epsilon}}{N(\fraka)^3}+\frac{x^{1/2+\epsilon}}{N(\fraka)^2}.
\end{equation}

In order to get good bounds for non-spherical representations we will require a more general counting function. We decompose $\{2,\ldots,d\}=J_1\cup J_2\cup J_3$ into three disjoint subsets. For $k\in \bbR^{J_1}$ with $k_j\geq 1$ and $\eta\in \bbR^{J_2}$ with $0<\eta_j\leq 2$ we define the counting function
\begin{equation}
N(x;k,\eta,\fraka)=\sharp\left\lbrace\begin{array}{ccl}
&\vline& \norm{\alpha}_1^2\leq x,\;\forall j\in J_1,\;|\Tr(\iota_j(\alpha))|<2     \\
 \alpha\in\calR^1(\fraka)&\vline & \forall\; j\in J_1,\; k\leq \norm{\alpha}_j^2\leq k+1\\
\alpha\neq  1 &\vline &\forall\;j\in J_2,\; 2-\eta\leq |\Tr(\iota_j(\alpha))|<2     \\
&\vline& \forall\;j\in J_2\cup J_3,\;\norm{\alpha}_j^2\leq C \end{array}\right\rbrace.
\end{equation}

Before estimating this counting function we prove a couple of lemmas counting points in $\calO_L$.
\begin{lem}\label{l:box}
Let $B\subset\bbR^n$ denote a box parallel to the axes and let $t_0\in \calO_L$.
Then for any ideal $\fraka\subset\calO_L$ we have
\[\#\{t\in \calO_L|t\equiv t_0\pmod{\fraka},\;\iota(t)\in B\}\leq \frac{\vol(B)}{N(\fraka)}+1,\]
where $\iota(t)=(\iota_1(t),\ldots,\iota_n(t))\in \bbR^n$ and $N(\fraka)=\# \calO_L/\fraka$.
\end{lem}
\begin{proof}
Without loss of generality we may assume that $\vol(B)<N(\fraka)$ and show that there is at most one such point.
Assume that both $\iota(t_1),\iota(t_2)\in B$ and $t_1\equiv t_2\pmod{\fraka}$. Since both $\iota(t_1),\iota(t_2)\in B$, then $$|N_{L/\bbQ}(t_1-t_2)|=\prod_j|\iota_j(t_1)-\iota_j(t_2)|\leq \vol(B)<N(\fraka).$$
On the other hand, since $t_1-t_2\in \fraka$, then $N_{L/\bbQ}(t_1-t_2)\in N(\fraka)\bbZ$ implying that $N_{L/\bbQ}(t_1-t_2)=0$ and hence $t_1=t_2$.
\end{proof}

\begin{lem}\label{l:quadratic}
For any $t\in \calO_L$
\[\sharp\{a,b\in\calO_L|a^2+b^2=t\}=O_\epsilon(|N_{L/\bbQ}(t)|^\epsilon).\]
\end{lem}
\begin{proof}
This is essentially \cite[Lemma 3.2]{SarnakXue91}; however, since we do not assume uniform bounds on $\iota_j(a),\iota_j(b)$ we will include the proof. Let $F=L(i)$ be a quadratic imaginary extension. We can write $a^2+b^2=(a+bi)(a-bi)$ and hence any solution to $a^2+b^2=t$ gives rise to an ideal factorization  $(t)=\fraka\bar\fraka$ with $\fraka$ the ideal generated by $a+bi$ in $\calO_F$ and $\bar\fraka$ its conjugate. The number of all ideal factorization $(t)=\fraka\mathfrak{b}$ is bounded by $O_\epsilon(|N_{L/\bbQ}(t)|^\epsilon)$. It remains to show that for a given decomposition $(t)=\fraka\bar\fraka$ the number of points $a,b\in \calO_L$ with $a^2+b^2=t$ and $(a+bi)=\fraka$ is uniformly bounded. Indeed, if $\tilde{a},\tilde{b}$ is another such pair, then $a+bi=u(\tilde{a}+\tilde{b}i)$ with $u$ some unit in $\calO_K$. Moreover, since $a^2+b^2=\tilde{a}^2+\tilde{b}^2$ then the unit $u$ satisfies that $u\bar{u}=1$, and there are only finitely such units in $F$.
\end{proof}

\begin{prop}\label{p:count}
For $k\in \bbR^{J_1},\eta\in \bbR^{J_2}$ as above let $|k|=\prod_{j\in J_1}(k_j+1)$ and $|\eta|=\prod_{j\in J_2} |\eta_j|$. Then
$$N(x;k,\eta,\fraka)\ll_\epsilon |\eta|\left( \frac{x^{1+\epsilon}}{N(\fraka)^3}+\frac{|k|^{\epsilon}x^{1/2+\epsilon}}{N(\fraka)^2}\right),$$
and in the special case where $J_1=\emptyset$ and $J_2\neq \emptyset$ we get a slightly better bound
$$N(x;\eta,\fraka)\ll_\epsilon \frac{|\eta|x^{1+\epsilon}}{N(\fraka)^3}.$$
\end{prop}
\begin{proof}
For $\alpha=a+bI+cJ+dK\in \calR^1(\fraka)$ we have
$$\norm{\alpha}_j^2\asymp |\iota_j(a)|^2+|\iota_j(b)|^2+ |\iota_j(c)|^2+|\iota_j(d)|^2.$$
The condition $n_\calA(\alpha)=a^2-pb^2-qc^2+pqd^2=1$ implies that $\norm{\alpha}_j^2\asymp |\iota_j(a)|^2+|\iota_j(c)|^2$ for $j\leq d$ and that $\norm{\alpha}_j$ is uniformly bounded for $j>d$.

Next, since $c,b,d\in \fraka$ we have that
$$(a+1)(a-1)=a^2-1=p(b^2+qd^2)+qc^2\in \fraka^2.$$
Since $a-1\in \fraka$ then $a+1\not\in \fraka$ and hence $a-1\in \fraka^2$
(we assume without loss of generality that $2\not\in \fraka$ as replacing $\fraka$ by $\frac{\fraka}{(2,\fraka)}$ will just change the implied constant).
We can thus bound our counting function by the sum
\begin{eqnarray*}
N(x;k,\fraka)&\ll &\mathop{\mathop{\mathop{\sum_{a-1\in\fraka^2}}_{\iota_1(a)\leq \sqrt{x}}}_{|\iota_j(a)|<1}}_{1-\eta_j\leq |\iota_j(a)|<1}
\mathop{\mathop{\mathop{\sum_{c\in\fraka}}_{|\iota_1(c)|\leq \sqrt{x}}}_{\sqrt{k}\leq \iota_j(c)\leq \sqrt{1+k}}}_{|\iota_j(c)|\leq C}\mathop{\mathop{\sum_{b,d\in \fraka}}_{p(b^2-qd^2)=}}_{a^2-qc^2-1}1\\
\end{eqnarray*}
By Lemma \ref{l:quadratic}, the inner sum is bounded by $O_\epsilon((x|k|)^\epsilon)$, and
by Lemma \ref{l:box}, the number of elements in the middle sum is bounded by $1+O(\frac{\sqrt{x}}{\sqrt{|k|}N(\fraka)})$. We can bound the number of elements in the outer sum by
$O(\frac{\sqrt{x}|\eta|}{N(\fraka^2)})$ (to do this, first add $a=1$ to the sum and get a bound of $1+O(\frac{\sqrt{x}|\eta|}{N(\fraka^2)})$ and then subtract one because we are not summing over $a=1$).
Combining these bounds we get
\[N(x;k,\eta,\fraka)\ll_\epsilon |\eta|\left(\frac{x^{1+\epsilon}}{N(\fraka)^3|k|^{1/2-\epsilon}}+\frac{|k|^\epsilon x^{\frac{1}{2}+\epsilon}}{N(\fraka)^2}\right).\]

Now, assume that $J_2\neq\emptyset$. If  $\alpha=a+bI+cJ+dK\in \calR^1(\fraka)$ satisfies the above condition, then we must have that $c\neq 0$. Indeed, if $c=0$, then for $j\in J_2$, we have $1=\iota_j(a^2-pb^2+pqd^2)\leq \iota_j(a)^2$ in contradiction to the constraint $|\iota_j(a)|<1$. Hence, in this case in the middle sum we are not summing over $c=0$.
If in addition $J_1=\emptyset$, then the middle sum is over $0\neq c\in \fraka$ such that
$(\iota_1(c),\ldots,\iota_n(c))\in [-\sqrt{x},\sqrt{x}]\times [-C,C]^{n-1}$. The number of such $c$'s including $c=0$ is bounded by $O(\frac{\sqrt{x}}{N(\fraka)})+1$ and since we are not summing over $c=0$ we get the bound $O(\frac{\sqrt{x}}{N(\fraka)})$.

\end{proof}

\section{The Selberg Trace Formula}
The upper bound in Theorem \ref{t:multiplicity1} is obtained via an application of the Selberg trace formula.
We refer to \cite[Sections 1-6]{Efrat87}, \cite[Chapter 3]{Hejhal76}
and \cite{Selberg95} for the full derivation of the trace formula in
this setting. For the readers convenience we recall here some basic facts and notations.

\subsection{Spherical transform}
For $m\in \bbZ$, let $\chi_m$ be the character of $K$ given by $\chi_m(k_\theta)=e^{im\theta}$.
We denote by $\calF_m\subset C^{\infty}(G)$ the space of functions satisfying
\[\forall\; k,k'\in K,\quad f(kgk')=\chi_m(kk')f(g).\]
The spherical transform on $\calF_m$ is defined by the integral
\begin{equation}\label{e:spherical}
 Sf(r)=\int_{G}f(g)\phi_{-m,\frac{1}{2}+ir}(g)dg=\int_{\bbH} f(p_z)\Im(z)^{\frac{1}{2}+ir}dz
\end{equation}
(when it converges), where $\phi_{m,s}\in \calF_m$ denoted the unique eigenfunction of the Casimir operator with eigenvalue $s(1-s)$ satisfying $\phi_{m,s}(1)=1$.
If $f\in\calF_m$ is supported on $B_\delta$, then $h(r)=Sf(r)$ is an even holomorphic function of uniform exponential type $\delta$, and the spherical transform gives a bijection between the space of compactly supported functions in $\calF_m$ and the space $\rm{PW}(\bbC)$ of even holomorphic functions of uniform exponential type.

For $f_1,f_2\in\calF_m$ their convolution $f_1*f_2\in\calF_m$ is given by
\[f_1*f_2(x)=\int_Gf_1(g)f_2(g^{-1}x)dg.\]
Under convolution the spherical transform satisfies
$$S(f_1*f_2)(r)=Sf_1(r)Sf_2(r).$$
Denote by $\check{f}(g)=\overline{f(g^{-1})}$, then $S(\check{f})(r)=\overline{Sf(\bar{r})}$; in particular, for any $f\in\calF_m$, we have that $S(f*\check{f})(r)=|Sf(r)|^2$ is positive on $\bbR\cup i\bbR$.

For $|m|>1$, let $\Phi_m=\phi_{m,|m|}$ denote the spherical function of weight $m$ and eigenvalue $|m|(1-|m|)$. This function is given by the formula
\begin{equation}\label{e:mSpherical}
\Phi_m(ka_tk')=\frac{\chi_m(kk')}{(\cosh\frac{t}{2})^{2|m|}}.
\end{equation}
By orthogonality relations, its spherical transform $S\Phi_m(r)=0$ unless $r=\pm i(|m|-\tfrac{1}{2})$ in which case it equals
$$S\Phi_m(i(|m|-\tfrac{1}{2})=\int_{G} |\Phi_m(g)|^2=\frac{4\pi}{2|m|-1}.$$
Note that when $|m|>1$, this function decays sufficiently fast so that the spherical transform absolutely converges.

\subsection{Trace formula}
Fix a weight $m=(m_1,\ldots,m_d)\in \bbZ^d$, and let $\chi_m(k)=\chi_{m_1}(k_1)\cdots\chi_{m_d}(k_d)$ denote the corresponding character of $K^d$. Let $L^2(\Gamma\bs G^d,m)$ denote the subspace of functions in $L^2(\Gamma\bs G)$  satisfying
\[\psi(gk)=\chi_m(k)\psi(g),\; \forall\; k\in K^d.\]
For $\lambda=(\lambda_1,\ldots,\lambda_d)\in \bbR^d$ let $V_\lambda(\Gamma,m)\subset L^2(\Gamma\bs G^d,m)$ denote the eigenspace
\[V_\lambda(\Gamma,m)=\{\psi\in L^2(\Gamma\bs G^d,m)|\Omega_j\psi+\lambda_j\psi=0\},\]
where $\Omega_j$ denotes the Casimir operator of $G$ acting on the $j$'th factor.

For $j=1,\ldots,d$ let $F_j\in\calF_{m_j}$ with $h_j=SF_j$ its spherical transform and let $F(g)=\prod_jF_j(g_j)$ and $h(r)=\prod_jh_j(r_j)$.
Taking trace in $L^2(\Gamma\bs G^d,m)$ of the kernel
$$F_\Gamma(g,g')=\sum_{\gamma\in \Gamma}F(g^{-1}\gamma g')$$
yields the trace formula
\begin{eqnarray}\label{e:trace}
\lefteqn{\sum_k \dim(V_{\lambda_k}(\Gamma,m))h(r_k)=\sum_{\gamma}\int_{\Gamma\bs G^d}F(g^{-1}\gamma g)dg=}\\
\nonumber =&&\vol(\Gamma\bs G)F(1)+\sum_{\{\gamma\}}\vol(\Gamma_\gamma\bs G^d_\gamma)\int_{G_\gamma^d\bs G^d} F(g^{-1}\gamma g)dg,
\end{eqnarray}
where the sum is over the set of all eigenvalues $\lambda_k=\frac{1}{4}+r_{k}^2$ of eigenfunctions in $L^2(\Gamma\bs G^d,m)$.

\subsection{Specialization}
To a representation $\pi=\pi_1\otimes\cdots \otimes\pi_d$ we attach the eigenvalue $\lambda=\lambda(\pi)\in\bbR^d$, such that the Casimir operator of $G$ acts on $\pi_j$ via multiplication by $\lambda_j$, specifically, $\lambda_j=s_j(1-s_j)$ when $\pi_j\cong \pi_{s_j}$ and $\lambda_j=|m_j|(1-|m_j|)$ when $\pi_j=\frakD_{m_j}$.
Let $\tilde{m}\in \bbZ^d$ such that $\tilde{m}_j\geq m_j$ (respectively $\tilde{m}_j\leq m_j$) whenever $m_j>0$ (respectively $m_j<0$). Since in any irreducible representation $\pi_j$ there is a unique (up to scaling) vector of $K$-type $\tilde{m}_j$, we have that
$$m(\pi,\Gamma)=\dim(V_\lambda(\Gamma,\tilde{m})).$$

In order to estimate $m(\pi,\Gamma)$ when $\pi$ is non-spherical in some factors we can use the trace formula, making sure that the $\tilde{m}_j$ are sufficiently large. However, this is rather wasteful and so we specialize our test functions to pick up specifically non-spherical representations of a prescribed type.

Assume that $\pi=\pi_1\otimes\cdots\otimes \pi_d$ with $\pi_j$ spherical for $j\leq d_0$ and $\pi_j\cong \frakD_{m_j}$ for $j\geq d_0$ with $|m|_j>1$ and let $m=(0,\ldots,0,m_{d_0+1},\ldots,m_d)\in\bbZ^d$. Now, for $j\leq d_0$ take $F_j\in \calF_0$ compactly supported and for $j>d_0$ take $F_j=\frac{2|m|-1}{4\pi}\Phi_{m_j}$ with $\Phi_{m_j}$ denoting the $m_j$-spherical function defined in (\ref{e:mSpherical}). With this choice of test function, the left-hand side of the trace formula is given by
\[\sum_{k} m(\pi_k,\Gamma)h(r_k(m)),\]
where the sum is over all representation $\pi_k=\pi_{k,1}\otimes\cdots\otimes\pi_{k,d}$ such that for $j\leq d_0$ $\pi_{k,j}$ is spherical with parameter $s=\frac{1}{2}+ir_{k,j}(m)$ and for $j>d_0$ and $\pi_{k,j}=\frakD_{m_j}$.

The left-hand side can be further expressed in terms of the functions $h_j=SF_j$ as follows: We have $F(1)=\prod_jF_j(1)$ where
\begin{equation}\label{e:trivial}
F_j(1)=\frac{1}{4\pi}\int_\bbR h_j(r)r\tanh(\pi r)dr,
\end{equation}
for  $j\leq d_0$ and  $F_j(1)=\frac{2|m_j|-1}{4\pi}$ for $d_0<j\leq d$. For the nontrivial conjugacy classes
$$\int_{G^d_\gamma\bs G^d}F(g^{-1}\gamma g)dg=\prod_j \int_{G_{\gamma_j}\bs G}F_j(g^{-1}\gamma_j g)dg,$$
where for $j\leq d_0$
\begin{equation}\label{e:Htransform1}
\int_{G_{\gamma_j}\bs G}F_j(g^{-1}\gamma_j g)dg=\left\lbrace\begin{array}{cc} \frac{\hat{h}_j(l)}{2\sinh(l/2)} & \gamma_j\sim a_{l}\\&\\
\frac{1}{2}\int_\bbR \hat{h}_j(u)\frac{\cosh(u/2)}{\cosh(u)-\cos(\theta)}du & \gamma_j\sim k_{\theta}\end{array}\right.\end{equation}
and for $d_0<j\leq d$
\begin{equation}\label{e:Htransform2}\frac{2|m_j|-1}{4\pi}\int_{G_{\gamma_j}\bs G}\Phi_{m_j}(g^{-1}\gamma_j g)dg=\left\lbrace\begin{array}{cc}
\frac{e^{im_j\theta}}{1-e^{\rm{sgn}(m_j) i\theta}} & \gamma_j\sim k_{\theta}\\&\\
0& \gamma_j\sim a_l \end{array}\right.\end{equation}
In particular, the only contribution to the trace formula comes from lattice points $\gamma\in \Gamma$ satisfying that $|\Tr(\gamma_j)|<2$ for $j>d_0$.
\begin{rem}
There is an alternative way to obtain this formula which is applicable also when some $|m_j|=1$. We note that if all $|m_j|=1$, there is an additional term of $(-1)^{d-d_0}h(i/2)$ entering on the spectral side (see \cite{kelmer10,Selberg95}).
\end{rem}

\section{Proof of Theorem \ref{t:multiplicity1}}
In this section, we obtain non-trivial upper bounds on the multiplicities of non-tempered representations occurring in $L^2(\Gamma(\fraka)\bs G^d)$, thus proving Theorem \ref{t:multiplicity1}.
For the sake of simplicity we write down the complete details in the case $d=2$. In this case there are two possibilities: either $\pi$ is spherical $\pi=\pi_{s_1}\otimes\pi_{s_2}$ with $s_1=\frac{1}{p(\pi)}\in(0,\frac{1}{2})$ and $s_2\in \tfrac{1}{2}+i\bbR^+$, or it is non-spherical $\pi=\pi_{s_1}\otimes \frakD_m$ for some $m\in \bbZ$. We will prove the theorem in each case separately.
\subsection{Some estimates}
Before we proceed with the proof we collect some estimates that we will need.
\begin{lem}\label{l:convolution}
Let $f\in \calF_0$ be supported on the ball $B_R$ and satisfy $|f(g)|\leq 1$ there.
Then $f*\check{f}$ is supported on $B_{2R}$ and satisfies $|f*\check{f}(g)|\ll e^{R-H(g)/2}$.
\end{lem}
\begin{proof}
Using the $KA^+K$ decomposition,
\begin{eqnarray*}
|f*\check{f}(a_t)|&\leq &\int_0^{2\pi}\int_0^\infty |f(a_{t'})||f(a_{-t}k_\theta a_{t'})|\sinh(t')dt'd\theta\\
&\leq & \int_0^R\sinh(t')\left(\int_{H(a_{-t}k_\theta a_{t'})\leq R}d\theta\right)dt'.
\end{eqnarray*}
From the equality $2\cosh(H(g))=\Tr(g^tg)$ we get that
\[\cosh(H(a_{-t}k_\theta a_{t'}))=\sin^2(\theta/2)\cosh(t+t')+\cos^2(\theta/2)\cosh(t-t').\]
Now if $t\geq 2R$ then $t-t'>R$ and hence $H(a_{-t}k_\theta a_{t'})\geq R$ for all $t'\leq R$.
When $t\leq 2R$ and $t'>R-t$ the condition $H(a_{-t}k_\theta a_{t'})\leq R$ is equivalent to
\[\sin^2(\theta)\leq \frac{\cosh(R)-\cosh(t-t')}{\cosh(t+t')-\cosh(t-t')}\asymp e^{R-t-t'}.\]
The inner integral is hence bounded by
\[\int_{H(a_{-t}k_\theta a_{t'})\leq R}d\theta\ll \int_{\sin^2(\theta)\leq e^{R-t-t'}}d\theta\ll\left\lbrace\begin{array}{cc} 1 & t'\leq R-t\\ e^{\frac{R-t-t'}{2}} & t'>R-t\end{array}\right.\]
We thus get that for $0\leq t\leq R$,
\[|f*\check{f}(a_t)|\ll \int_0^{R-t} e^{t'}dt+e^{\frac{R-t}{2}}\int_{R-t}^R e^{t'/2}dt'\leq e^{R-t/2},\]
and for $R\leq t\leq 2R$
\[|f*\check{f}(a_t)|\ll e^{\frac{R-t}{2}}\int_{0}^R e^{t'/2}dt'\leq e^{R-t/2}.\]
\end{proof}

%\begin{lem}\label{l:spherical}
%Let $f\in \calF_0$ be supported on $B_R$ and satisfy $0\leq f(g)\leq 1$ and $f(g)=1$ on $B_{R-1}$.
%Then for $\tau\in (0,\frac{1}{2})$ we have $|Sf(i\tau)|\asymp e^{(\tau+1/2)R}$.
%\end{lem}
%\begin{proof}
%Write
%\begin{eqnarray*}Sf(i\tau)&=&\int_G f(g)\phi_{0,\frac{1}{2}-\tau}(g)dg=2\pi\int_0^{R}f(a_t)\phi_{0,\frac{1}{2}-\tau}(a_t)\sinh(t)dt\end{eqnarray*}
%For $s\in(0,\frac{1}{2})$ real the spherical function $\phi_{0,s}(a_t)$ is real, positive, and decays like $\phi_{0,s}(a_t)\asymp e^{-st}$.
%We thus get that
%\begin{eqnarray*}
%\int_0^{R}f(a_t)\phi_{0,\frac{1}{2}-\tau}(a_t)\sinh(t)dt\geq\int_0^{R-1}\phi_{0,\frac{1}{2}-\tau}(a_t)\sinh(t)dt\gg e^{(\tau+\frac{1}{2})R}.
%\end{eqnarray*}
%and
%\begin{eqnarray*}
%\int_0^{R}f(a_t)\phi_{0,\frac{1}{2}-\tau}(a_t)\sinh(t)dt\leq \int_0^{R}\phi_{0,\frac{1}{2}-\tau}(a_t)\sinh(t)dt\ll e^{(\tau+\frac{1}{2})R}.
%\end{eqnarray*}
%\end{proof}

\begin{lem}\label{l:Bm}
For any $m\in\bbZ$ let $\Phi_m=\phi_{m,|m|}$ denote the $m$-spherical function with eigenvalue $|m|(1-|m|)$.
Let $\Omega\subset G$ denote a compact set, then for any $\gamma\in G$ with $|\Tr(\gamma)|<2-\tfrac{1}{\sqrt{m}}$ we have
\[\int_{\Omega}|\Phi_m(g^{-1}\gamma g)|dg\leq \frac{\log_2{m}}{\sqrt{m}}+ \frac{\vol(\Omega)}{m}.\]
\end{lem}
\begin{proof}
Let $\tau\in G$ such that $\gamma=\tau^{-1}k_\theta \tau$. Fix a small parameter $\eta$ and separate $\Omega=\Omega_{\eta}\cup\Omega_{\eta}^c$ where
$$\Omega_{\eta}=\Omega\cap\{\tau ka_tk'|0\leq t\leq \eta,k,k'\in K\}.$$
For $g\in\Omega_\eta^c$, we can write $g=\tau ka_tk'$ with $t>\eta$, and hence,
$H(g^{-1}\gamma g)=H(a_{-t}k_\theta a_t)$. Now
\[\cosh(H(a_{-t}k_\theta a_t))=\tfrac{1}{2}\norm{a_{-t}k_\theta a_t}^2=1+\sin^2(\tfrac{\theta}{2})(\cosh(2t)-1).\]
The condition $|\Tr(\gamma)|<2-\tfrac{1}{\sqrt{m}}$ implies that $\sin^2(\theta/2)>\frac{1}{\sqrt{m}}$ and the condition $t>\eta$ implies that $\cosh(2t)-1>2\eta^2$. Taking $\eta^2=\frac{\log_2{m}}{\sqrt{m}}$ we get that on $\Omega_{\eta}^c$,
\[|\Phi_m(g^{-1}\gamma g)|=(\cosh^2(\frac{H(a_{-t}k_\theta a_t)}{2}))^{-m}\leq (1+\frac{\log_2{m}}{m})^{-m}\leq \frac{1}{m}.\]
The main contribution then comes from $\Omega_{\eta}$ and is bounded by
\[\int_{\Omega_\eta}|\Phi_m(g^{-1}\gamma g)|dg\leq\int_{\Omega_\eta}dg\leq\int_0^\eta\sinh(t)dt\leq\eta^2.\]
\end{proof}
\subsection{Proof for spherical case}
Let $\pi=\pi_{s_1}\otimes \pi_{s_2}$ with $s_1=\frac{1}{2}-\tau\in(0,\tfrac{1}{2})$ $s_2=\tfrac{1}{2}+it$ so that $T(\pi)\asymp |t|+1$ and $\tau=\frac{1}{2}-\frac{1}{p(\pi)}$.

Let $f_1,f_2\in \calF_0$ be supported on $B_R$ and $B_{1/2}$ respectively, with $0\leq f_j(g)\leq 1$ and $f_1(g)=1$ on $B_{R-1}$ and let $h_j=SF_j$ denote their spherical transforms.
Let $F_1(g)=f_1*\check{f}_1$ and $F_2(g)=f_2*\check{f_2}$ so that their spherical transforms $SF_j=|h_j|^2$ are positive on $\bbR\cup i\bbR$. Following the same arguments of \cite{SarnakXue91}, we can use the trace formula with the test function $F(g)=F_1(g_1)F_2(g_2)$ together with the counting estimate (\ref{e:counting1}) to get the multiplicity bound. However, due to the fast decay of the spherical transform $h_2(r)$ as $r\to\infty$ using such a test function directly will produce a very poor dependence on the parameter $|t|\asymp T(\pi)$.

In order to get a better dependence we consider the trace formula with a shifted test function.
Further assume that $h_2$ itself is positive and normalized so that $h_2(0)=1$. We consider the shifted test function
$h_{2,t}(r)=\frac{h_2(r+t)+h_2(r-t)}{2}$. In the following lemma we control how the trace formula changes when we replace $|h_2|^2$ with $|h_{2,t}|^2$:

\begin{lem}\label{l:translate}
Let $f_1,f_2,F_1,F_2\in \calF_0$ as above with $h_j=Sf_j$ their spherical transforms.
Let $h_{2,t}(r)=\frac{h_2(r+t)+h_2(r-t)}{2}$ and $h_t(r)=h_1(r_1)h_{2,t}(r_2)$.
Then
\begin{eqnarray*}
\sum_{k} m(\pi_k,\Gamma)|h_t(r_{k})|^2 &\ll& \vol(\Gamma\bs G^d)F_1(1)(|t|+1)\\
&+&{\sum_{\gamma\neq 1}}\int_{\Gamma\bs G^d}|F(g^{-1}\gamma g)|dg.
\end{eqnarray*}
\end{lem}
\begin{proof}
Let $f_{2,t}=S^{-1}h_{2,t}$ denote the inverse spherical transform of the shifted function and let $F_{2,t}=f_{2,t}*\check{f}_{2,t}$. The trace formula for the shifted function then reads

\begin{eqnarray*}
\lefteqn{\sum_{k} m(\pi_k,\Gamma)|h_t(r_k)|^2=\vol(\Gamma\bs G^d)F_1(1)F_{2,t}(1)}\\
&&+\sum_{\{\gamma\}\neq 1}\vol(\Gamma_\gamma\bs G^d_\gamma)\int_{G_{\gamma_1}\bs G} F_{1}(g^{-1}\gamma_1 g)dg\int_{G_{\gamma_2}\bs G} F_{2,t}(g^{-1}\gamma_2 g)dg.
\end{eqnarray*}
Using (\ref{e:trivial}), we can bound
\begin{eqnarray*}
F_{2,t}(1)&=&\frac{1}{2\pi}\int_{\bbR} |h_{2,t}(r)|^2r\tanh(\pi r) dr\\
&\ll& |t|\int_\bbR |h_2(r)|^2dr+\int_\bbR |h_2(r)|^2|r|dt\ll |t|+1
\end{eqnarray*}
Next, note that the Fourier transform $\hat{h}_{2,t}$ satisfies
$$|\hat{h}_{2,t}(u)|=|\hat{h}_2(u)\cos(tu)|\leq |\hat{h}_2(u)|=\hat{h}_2(u),$$
where the positivity of $\hat{h}_2(u)$ follows from the positivity of $f_2$ and (\ref{e:Htransform1}).
Consequently, we also have that the Fourier transform of $|h_{2,t}|^2$ (which is the convolution of $\hat{h}_{2,t}$ with itself) is bounded by the Fourier transform of
$|h_2|^2$. Hence, from the explicit formulas (\ref{e:Htransform1}) for the period integrals, we see that in both the hyperbolic and elliptic cases,
\[\big|\int_{G_{\gamma_2}\bs G} F_{2,t}(g^{-1}\gamma_2 g)dg\big|\leq \int_{G_{\gamma_2}\bs G} F_{2}(g^{-1}\gamma_2 g)dg= \int_{G_{\gamma_2}\bs G} |F_{2}(g^{-1}\gamma_2 g)|dg.\]

We thus get that
\begin{eqnarray*}
\sum_{k} m(\pi_k,\Gamma)|h_t(r_k)|^2&\ll& \vol(\Gamma\bs G)F_1(1)(|t|+1)\\
&+&\sum_{\{\gamma\}}\int_{\G_{\gamma}\bs G^d} |F(g^{-1}\gamma g)|dg.
\end{eqnarray*}
Unfolding the sum over the conjugacy classes back to a sum over the lattice elements
concludes the proof.
\end{proof}
\begin{rem}
We remark that the function $F(g)$ above is positive so that the absolute value on the right hand side is redundant.
However, when $d>2$, we will use this lemma when in some factors the functions are not assumed to be positive.
We take the absolute value in order to make the statement (and the proof) applicable to that setting as well.
\end{rem}

We now proceed with the proof of Theorem \ref{t:multiplicity1}.
With $F_1,F_2,h_1,h_2,h_{2,t}$ as above, from the trace formula and the positivity of $|h_1|^2$, $|h_{2,t}|^2$,
and the observation that $|h_{2,t}(t)|\geq \tfrac{1}{2}$ we have
\begin{eqnarray*}
m(\pi,\Gamma(\fraka))|h_1(i\tau)|^2&\leq& 4\sum_k m(\pi_k,\Gamma(\fraka))|h_t(r_k)|^2\\
&\ll&  (|t|+1)V(\fraka)F_1(1)+ \sum_{1\neq\gamma\in \Gamma(\fraka)}\int_{\Gamma(\fraka)\bs G^2}|F(g^{-1}\gamma g)|dg\\
&=& V(\fraka)\bigg((|t|+1)F_1(1)+\sum_{1\neq \gamma\in \Gamma(\fraka)}\int_{\Gamma\bs G^2}|F(g^{-1}x g)|dg\bigg).\\
%&\leq & V(\fraka)\sum_{\gamma\in \Gamma(\fraka)}\int_{B_{\frac{\delta}{2}}F_1(g_1^{-1}\gamma_1 g_1)dg\int_{B_{\frac\delt 2}}F_2(g_2^{-1}\gamma_2 g_2)dg_2\\
\end{eqnarray*}
By Lemma \ref{l:convolution}, we can bound $F_1(1)\ll e^{R}$. The sum over the nontrivial elements now has no dependence on $t$, and we can proceed as in \cite{SarnakXue91} to get
\begin{equation}\label{e:lsum1}
\sum_{1\neq \gamma\in \Gamma(\fraka)}\int_{\Gamma\bs G^2}F(g^{-1}\gamma g)dg
\ll_\epsilon  e^R(\frac{e^{R(1+\epsilon)}}{V(\fraka)}+\frac{e^{2\epsilon R}}{V(\fraka)^{2/3}}).
\end{equation}
Indeed, by Lemma \ref{l:convolution}, we can bound
\begin{eqnarray*}
\int_{\Gamma\bs G^2}|F(g^{-1}x g)|dg%&\leq& \int_{B_\delta}F_1(g_1^{-1}\gamma_1 g_1)dg\int_{B_\delta}F_2(g_2^{-1}\gamma_2 g_2)dg_2\\
&\ll& \left\{\begin{array}{cc} e^{R-H(\gamma_1)/2}& H(\gamma_1)\leq 2R,H(\gamma_2)\leq 1\\&\\
0& \mbox{otherwise.}\end{array}\right..
\end{eqnarray*}
Summing over all lattice points, we get
\begin{eqnarray*}
\sum_{1\neq \gamma\in \Gamma(\fraka)}\int_{\Gamma\bs G^2}F(g^{-1}\gamma g)dg
&\ll& e^R\mathop{\sum_{\gamma\in \Gamma(\fraka)}}_{ H(\gamma_1)\leq 2R,H(\gamma_2)\leq 1}e^{-H(\gamma_1)/2}\\
&\ll& e^R\int^{2R} e^{-t/2}d\tilde{N}(t,\fraka)\\
&\ll& e^R \int^{2R} e^{-t/2}\tilde{N}(t,\fraka)dt\\
\end{eqnarray*}
with $\tilde{N}(t;\fraka)=N(e^t;\fraka)$, and the bound $\tilde{N}(t;\fraka)\ll_\epsilon \frac{e^{t(1+\epsilon)}}{V(\fraka)}+\frac{e^{t(1/2+\epsilon)}}{V(\fraka)^{2/3}}$ gives (\ref{e:lsum1}).

We thus have
\begin{eqnarray*}
m(\pi,\Gamma(\fraka))|h_1(i\tau)|^2\ll_\epsilon  V(\fraka)e^R(|t|+1+\frac{e^{R(1+\epsilon)}}{V(\fraka)}+\frac{e^{2\epsilon R}}{V(\fraka)^{2/3}})
\end{eqnarray*}
For $s=\tfrac{1}{2}-\tau\in(0,\frac{1}{2})$ we have $|h_1(i\tau)|^2\asymp_s e^{(2\tau+1)R}$.
This follows from (\ref{e:spherical}) together with the positivity and rate of decay of the spherical function $\phi_{0,s}(a_t)\asymp_s e^{-st}$ (see also \cite[12.8]{Iwaniec02} for a precise formula).
Hence, dividing by $|h_1(i\tau)|^2$ we get
\begin{eqnarray*}
m(\pi,\Gamma(\fraka))\ll  e^{-2\tau R} (V(\fraka)(|t|+1)+e^{R(1+\epsilon)}+V(\fraka)^{1/3}e^{2\epsilon R})
\end{eqnarray*}
and setting $e^{R}=(|t|+1)V(\fraka)=T(\pi)V(\fraka)$ gives
\[m(\pi,\Gamma(\fraka))\ll_\epsilon (T(\pi)V(\fraka))^{\frac{2}{p(\pi)}+\epsilon}.\]
Since $\frac{2}{p(\pi)}\leq \frac{1}{2}+\frac{1}{p(\pi)}$ this concludes the proof for $\pi$ spherical.

\subsection{Proof for non-spherical case}
Let $\pi=\pi_{s_1}\otimes\frakD_m$ with $s_1=\frac{1}{2}-\tau\in(0,\tfrac{1}{2})$ so that $T(\pi)\asymp |m|$ and $\tau=\frac{1}{2}-\frac{1}{p(\pi)}$.
It is possible to proceed as in the spherical case with test function $F(g)=F_1(g)F_2(g)$ with $F_1$ as before and $F_2=f_2*\check{f_2}$ with $f_2\in \calF_m$ supported on $B_{1/2}$. However, as mentioned above this approach is rather wasteful and does not give a very good dependence on $|m|\asymp T(\pi)$.
Indeed, when $|m|>1$ the spherical transform of such a function satisfies $Sf_2(i(|m|-\tfrac{1}{2}))\ll \frac{1}{|m|-1}$ and hence following this approach will
give a result of the form $m(\pi,\Gamma(\fraka))\ll T(\pi)^2V(\fraka)^{\frac{2}{p(\pi)}+\epsilon}$.

In order to get a better dependence, when $|m|>1$ we use the test function $F_2=\frac{2|m|-1}{4\pi}\Phi_m$ with $\Phi_m$ the $m$-spherical function with eigenvalue $|m|(1-|m|)$ which picks out precisely the representations with $\frakD_m$ in the second factor. %This function is not compactly supported which makes the counting argument slightly more involved.
From the trace formula with $F(g)=F_1(g)F_2(g)$ with $F_1$ as before and $F_2=\frac{2|m|-1}{4\pi}\Phi_m$ we get
\begin{eqnarray*}
m(\pi,\Gamma(\fraka))|h_1(i\tau)|^2 &\ll&  V(\fraka)F(1)+ {\sum_{\gamma\in \Gamma(\fraka)}}'\int_{\Gamma(\fraka)\bs G^2}F(g^{-1}\gamma g)dg\\
&=&  V(\fraka)\bigg(F(1)+ {\sum_{\gamma\in \Gamma(\fraka)}}'\int_{\Gamma\bs G^2}F(g^{-1}\gamma g)dg\bigg)\\
&\leq& V(\fraka)\bigg(F(1)+{\sum_{\gamma\in \Gamma(\fraka)}}'\int_{B_{\delta/2}^2}|F(g_1^{-1}\gamma_1 g_1)|dg\bigg),\\
\end{eqnarray*}
where $\sum'$ indicates that we are summing over lattice points satisfying that $|\Tr(\gamma_2)|<2$ and $\delta$ is chosen so that $B_{\delta/2}^2$ contains a fundamental domain for $\Gamma\bs G^2$.

For any $g_j\in B_{\delta/2}$ and any $\gamma\in \Gamma(\fraka)$ by (\ref{e:triangle}) we have that
$$|H(g_j^{-1}\gamma_j g_j)- H(\gamma_j)|\leq \delta.$$
%\begin{eqnarray*}
%H(g_j^{-1}\gamma_jg_j)&=&d(\gamma_jg_ji,g_ji)\\
%&\leq&d(\gamma_jg_ji,\gamma_ji)+d(\gamma_ji,i)+d(i,g_ji)=H(\gamma_j)+2H(g_j),
%\end{eqnarray*}
%so $H(g_j^{-1}\gamma_jg_j)-H(g_j^{-1}\gamma_jg_j)\leq \delta$ and replacing $g_j$ with $g_j^{-1}$ gives the other direction.
We can thus bound
$$\int_{B_{\delta/2}}|F_1(g_1^{-1}\gamma_1 g_1)|dg_1\ll \left\lbrace\begin{array}{cc} e^Re^{-H(\gamma_1)/2}& H(\gamma_1)\leq 2R+\delta\\ 0& \mbox{otherwise}\end{array}\right.,$$
and by Lemma \ref{l:Bm}
$$\int_{B_{\delta/2}}|F_2(g_2^{-1}\gamma_2 g_2)|dg_2\ll_\epsilon \left\lbrace\begin{array}{cc}
|m|e^{m(\delta -H(\gamma_2))} & H(\gamma_2)>\delta\\
|m|^{1/2+\epsilon}  & |\Tr(\gamma_2)|<2-\tfrac{1}{\sqrt{m}}\\
|m| & \mbox{otherwise.}\end{array}\right.,$$
implying that
\begin{eqnarray}\label{e:mainsum}
{\sum_{\gamma\in \Gamma(\fraka)}}'\int_{B_\delta^2}F(g^{-1}\gamma g)dg &\ll_\epsilon& |m|\big(\mathop{\mathop{e^R{\sum_{\gamma\in \Gamma(\fraka)}}''}_{ H(\gamma_1)\leq 2R}}_{H(\gamma_2)\leq \delta}e^{-H(\gamma_1)/2}\big)\\
\nonumber +|m|^{1/2+\epsilon}\big(\mathop{\mathop{e^R{\sum_{\gamma\in \Gamma(\fraka)}}'}_{ H(\gamma_1)\leq 2R}}_{H(\gamma_2)\leq \delta}e^{-H(\gamma_1)/2}\big)
&+&|m|\big(e^{R+m\delta}{\mathop{\mathop{{\sum_{\gamma\in \Gamma(\fraka)}}'}_{H(\gamma_1)\leq 2R}}_{H(\gamma_2)>\delta}}e^{-H(\gamma_1)/2}e^{-mH(\gamma_2)}\big).
\end{eqnarray}
where the sum $\sum''$ is over lattice points satisfying that $2-\tfrac{1}{\sqrt{m}}\leq |\Tr(\gamma_2)|<2$ and the sum $\sum'$ is over lattice points with $|\Tr(\gamma_2)|<2$. We now bound each of these sums separately.

For the first sum, let
$\tilde{N}(t;\eta,\fraka)=N(e^t;\eta,\fraka)$ be the counting function defined in (\ref{e:counting1}) with $J_1=J_3=\emptyset$ and $J_2=\{2\}$.
As before after integrating by parts and inserting the bound $\tilde{N}(t;\tfrac{1}{\sqrt{m}},\fraka)=O(\frac{e^{t(1+\epsilon)}}{V(\fraka)\sqrt{m}})$ from Proposition \ref{p:count} we get
\begin{eqnarray}\label{e:sumI}
\mathop{\mathop{{\sum_{\gamma\in \Gamma(\fraka)}}''}_{ H(\gamma_1)\leq 2R}}_{H(\gamma_2)\leq \delta}e^{-H(\gamma_1)/2}\ll_\epsilon \frac{e^{R(1+\epsilon)}}{\sqrt{m}V(\fraka)}.
\end{eqnarray}
For the second sum the same argument gives
\begin{eqnarray}\label{e:sumII}
\mathop{\mathop{{\sum_{\gamma\in \Gamma(\fraka)}}'}_{ H(\gamma_1)\leq 2R}}_{H(\gamma_2)\leq \delta}e^{-H(\gamma_1)/2}\ll_\epsilon \frac{e^{R(1+\epsilon)}}{V(\fraka)}.
\end{eqnarray}
We are left with the third sum over lattice points with $H(\gamma_2)>\delta$.  Recall that $e^{H(\gamma_2)}\asymp \norm{\gamma_2}^2$ to get
\begin{eqnarray*}
\mathop{\mathop{{\sum_{\gamma\in \Gamma(\fraka)}}'}_{H(\gamma_1)\leq 2R}}_{H(\gamma_2)>\delta}e^{-H(\gamma_1)/2}e^{-mH(\gamma_2)}
&\ll&\sum_{k\geq e^\delta} \mathop{\mathop{{\sum_{\gamma\in \Gamma(\fraka)}}'}_{H(\gamma_1)\leq 2R}}_{k\leq \norm{\gamma_2}^2\leq k+1}e^{-H(\gamma_1)/2}e^{-mH(\gamma_2)}\\
&\ll& \sum_{k>e^\delta} k^{-m} \mathop{\mathop{{\sum_{\gamma\in \Gamma(\fraka)}}'}_{H(\gamma_1)\leq 2R}}_{k\leq \norm{\gamma_2}^2\leq k+1}e^{-H(\gamma_1)/2}.
\end{eqnarray*}
Now for each $k$ from Proposition \ref{p:count} with $J_1=\{2\}$ and $J_2=J_3=\emptyset$,
$$\tilde N(t;k,\fraka)=N(e^t;k,\fraka)\ll_\epsilon \frac{e^{t(1+\epsilon)}}{V(\fraka)}+ k^\epsilon\frac{e^{t(\frac{1}{2}+\epsilon)}}{V(\fraka)^{2/3}},$$
implying that
\begin{equation*}
\mathop{\mathop{{\sum_{\gamma\in \Gamma(\fraka)}}'}_{H(\gamma_1)\leq 2R}}_{k\leq \norm{\gamma_2}^2\leq k+1}e^{-H(\gamma_1)/2}%&=& \int^{2R}e^{-t/2}d\tilde{N}(t;k,\fraka)\\
%&\ll& \int^{2R}e^{-t/2}\tilde{N}(t;k,\fraka)dt\\
%&\ll& |k|^\epsilon\int^{2R}e^{-t/2} \left(\frac{e^{t(1+\epsilon)}}{V(\fraka)}+ %\frac{e^{t(\frac{1}{2}+\epsilon)}}{V(\fraka)^{2/3}}\right)dt\\
\ll_\epsilon \frac{e^{R(1+2\epsilon)}}{V(\fraka)}+k^{\epsilon}\frac{e^{2R\epsilon}}{V(\fraka)^{2/3}}.\\
\end{equation*}
Hence, summing over all $k$ and replacing $\sum_{k>e^\delta}k^{-m+\epsilon}\ll \frac{e^{-m\delta}}{m}$ we get
%\begin{eqnarray}\label{e:sumIII}
%\mathop{\mathop{{\sum_{\gamma\in \Gamma(\fraka)}}'}_{H(\gamma_1)\leq %2R}}_{H(\gamma_2)>\delta}e^{-H(\gamma_1)/2}e^{-mH(\gamma_2)}
%%&\ll& \sum_{k\geq e^\delta} k^{-m} \big(\mathop{\mathop{{\sum_{\gamma\in \Gamma(\fraka)}}'}_{H(\gamma_1)\leq 2R}}_{k\leq \norm{\gamma_2}^2\leq k+1}e^{-H(\gamma_1)/2}\big)\\
%%\nonumber &\ll&
%%\big(\frac{e^{R(1+2\epsilon)}}{V(\fraka)}+\frac{e^{2R\epsilon}}{V(\fraka)^{2/3}}\big)\big(\sum_{k\geq e^\delta} %k^{-m+\epsilon}\big)\\
% &\ll& \frac{e^{-m\delta}}{m}\big(\frac{e^{R(1+2\epsilon)}}{V(\fraka)}+\frac{e^{2R\epsilon}}{V(\fraka)^{2/3}}\big).
%\end{eqnarray}
\begin{equation}\label{e:sumIII}
\mathop{\mathop{{\sum_{\gamma\in \Gamma(\fraka)}}'}_{H(\gamma_1)\leq 2R}}_{H(\gamma_2)>\delta}e^{-H(\gamma_1)/2}e^{-mH(\gamma_2)}
\ll_\epsilon \frac{e^{-m\delta}}{m}\big(\frac{e^{R(1+2\epsilon)}}{V(\fraka)}+\frac{e^{2R\epsilon}}{V(\fraka)^{2/3}}\big).
\end{equation}
Inserting the bounds from (\ref{e:sumI}),(\ref{e:sumII}),(\ref{e:sumIII}) to (\ref{e:mainsum}) we get
\begin{equation*}
{\sum_{\gamma\in \Gamma(\fraka)}}'\int_{B_\delta^2}F(g^{-1}\gamma g)dg\ll_\epsilon |m|^{1/2+\epsilon}\frac{e^{2R(1+\epsilon)}}{V(\fraka)}+\frac{e^{R(1+\epsilon)}}{V(\fraka)^{2/3}}.
\end{equation*}
This gives us the bound
\[m(\pi,\Gamma(\fraka))|h_1(i\tau)|^2\ll_\epsilon |m|e^{R}V(\fraka)+m^{1/2+\epsilon}e^{2R(1+\epsilon)}+V(\fraka)^{1/3}e^{R(1+\epsilon)},\]
dividing by $|h_1(i\tau)|^2\asymp e^{(2\tau+1)R}$ %gives
%\[m(\pi,\Gamma(\fraka))\ll e^{-2R\tau}\big(mV(\fraka)+m^{1/2+\epsilon}e^{R(1+\epsilon)}+V(\fraka)^{1/3}e^{R\epsilon}\big),\]
and taking $R=\log(V(\fraka))+\tfrac{1}{2}\log(m)$ we get
\[m(\pi,\Gamma(\fraka))\ll_\epsilon
 |m|^{1-\tau+\epsilon}V(\fraka)^{1-2\tau+\epsilon}\asymp T(\pi)^{\frac{1}{2}+\frac{1}{p(\pi)}+\epsilon}V(\fraka)^{\frac{2}{p(\pi)}+\epsilon}.\]

\subsection{Proof for $d>2$}
The proof for $d>2$ follows from the same arguments as above.
We briefly go over the modifications needed in this case.
Let $\pi$ denote an irreducible representation of $G^d$ given by $\pi=\pi_1\otimes \cdots\otimes \pi_d$ with $\pi_j=\pi_{s_j}$ for $j\leq d_0$ and $\pi_j=\frakD_{m_j}$ for $j>d_0$ where
$s_1=\frac{1}{2}-\tau=\frac{1}{p(\pi)}\in (0,\tfrac{1}{2})$ and
$s_j=\frac{1}{2}+it_j$ for $2\leq j\leq d_0$. Then $T(\pi)\asymp \prod_{j=2}^{d_0}(|t_j|+1)\prod_{j>d_0}|m_j|$.

Let $F_1=f_1*\check{f_1}$ as above, for $2\leq j\leq d_0$ let $F_j=f_2*\check{f_2}$ where $f_2\in \calF_0$ is supported on $B_{1/2}$, and for $j>d_0$ let $F_j=\frac{2|m_j|-1}{4\pi}\Phi_{m_j}$. Let $h_1=Sf_1$ and for $2\leq j\leq d_0$ let $h_j(r_j)=\frac{Sf_2(r_j-t_j)+Sf_2(r_j+t_j)}{2}$ and $h(r)=\prod_j h_j(r_j)$. If $|m_j|=1$ for some $j$ we take $F_j=f_3*\check{f_3}$ with $f_3\in \calF_1$ supported on $B_{1/2}$ (instead of the spherical function $\Phi_{m_j}$).
Then by (the same proof of) Lemma \ref{l:translate},
\begin{eqnarray*}
m(\pi,\Gamma(\fraka))|h_1(i\tau)|\ll T(\pi)V(\fraka)F_1(1)+V(\fraka){\sum_{\gamma\in \Gamma}}'\int_{\Gamma(\fraka)\bs G}|F(g^{1-}\gamma g)|dg
\end{eqnarray*}
where the notation $\sum'$ indicates that we are summing only over lattice points satisfying $|\Tr(\gamma_j)|<2$ whenever $|m_j|>1$. For the sum over the lattice points,
the same arguments as in the proof of the non-spherical case (and some elementary combinatorics) gives
\[{\sum_{\gamma\in \Gamma(\fraka)}}'\int_{\Gamma\bs G}|F(g^{-1}\gamma g)|dg\ll_\epsilon |m|^{1/2+\epsilon}\bigg(\frac{e^{2R(1+\epsilon)}}{V(\fraka)}+\frac{e^{R(1+\epsilon)}}{V(\fraka)^{2/3}}\bigg),\]
with $|m|=\prod_{j>d_0}|m_j|$ and the result follows as before.
\begin{rem}
For $d=2$, there is an alternative (and very simple) proof for the multiplicity bound $m(\pi,\Gamma)\ll T(\pi)\vol(\Gamma\bs G^2)^{\frac{2}{p(\pi)}}$ for all lattices, and not just congruence covers. However, since this proof does not generalize for $d>2$ we leave it as Appendix \ref{s:nonspherical}.
\end{rem}

\section{Proof of Theorem \ref{t:multiplicity2}}
In this section, we prove Theorem \ref{t:multiplicity2} using an adaptation of the arguments in \cite{kelmerSarnak09} for the congruence cover $\Gamma(\fraka)$. We then use this to deduce Corollary \ref{c:upperbound} giving an upper bound on the parameter $T(\pi)$ of new representations.
\subsection{Proof of Theorem \ref{t:multiplicity2}}
For simplicity we write the proof for $d=2$ (see \cite[Section 3.3]{kelmerSarnak09} for the changes needed for $d>2$).
By (the proof of)  \cite[Proposition 3.1]{kelmerSarnak09} we get that for any $c>0$ and $T\leq T(\pi)\leq 2T$ (in particular, for $T=T(\pi)$)
\begin{eqnarray}\label{e:multbound1}
\lefteqn{\;\;\;\quad\quad m(\pi,\Gamma(\fraka))T^{c|\frac{1}{2}-\frac{1}{p(\pi)}|}}\\
\nonumber&&\ll_\epsilon V(\fraka)\bigg(
\frac{T^2}{\log(T)}+ T\!\!\!\!\!\!\!\!\!\!\mathop{\sum_{|t_1|\leq
T^{c/2}}}_{||t_2|-2|\leq T^{\epsilon-2}}\!\!\!\!\!\!\frac{F_\Gamma(t)}{\sqrt{
(t_1^2-4)(t_2^2-4)}}+\frac{1}{T}\!\!\!\!\mathop{\sum_{|t_1|\leq T^{c/2}}}_{|t_2|<
2}\!\frac{F_\Gamma(t)}{\sqrt{(t_1^2-4)(t_2^2-4)}}\bigg).
\end{eqnarray}
where the summation is over elements $t=(t_1,t_2)\in\Tr(\Gamma(\fraka))$ and
$$F_\Gamma(t)=\mathop{\sum_{\{\gamma\}}}_{\Tr(\gamma)=t}\vol(\Gamma_\gamma\bs
G_\gamma^2).$$

Since (a conjugate of) $\Gamma(1)$ is of finite index in a lattice $\Delta$ derived from a quaternion algebra we have that
$F_\Gamma(t)\ll F_{\Delta}(t)$ and by \cite[Proposition 3.3]{kelmerSarnak09} for such lattices we have $F_{\Delta}(t)\ll |(t_1^2-4)(t_2^2-4)|^{1/2+\epsilon}$.
Inserting this bound in (\ref{e:multbound1}) we get
\begin{equation}\label{e:multbound2}
m(\pi,\Gamma(\fraka))T^{c|\frac{1}{2}-\frac{1}{p(\pi)}|}\ll_\epsilon V(\fraka)\bigg(
T^2+T^{1+\epsilon}\!\!\!\!\!\!\!\!\mathop{\sum_{|t_1|\leq
T^{c/2}}}_{||t_2|-2|\leq T^{\epsilon-2}}\!\!\!\!\!\!\!\!\!1
\!+\!\frac{T^\epsilon}{T}\!\!\!\!\!\mathop{\sum_{|t_1|\leq T^{c/2}}}_{|t_2|<
2}\!\!\!\!1 \bigg).
\end{equation}
Finally we recall that for any $t=(t_1,t_2)\in \Tr(\Gamma(\fraka))$ there is $\alpha\in \Delta(\fraka)$ with $t_j=\iota_j(\Tr(\alpha))$ for $j=1,2$ and $|\iota_j(\Tr(\alpha))|<2$ for $j>2$.
As in the proof of Proposition \ref{p:count} we have that any $\alpha\in \Delta(\fraka)$ satisfies $\Tr(\alpha)\equiv 2\pmod {\fraka^2}$ and hence by Lemma \ref{l:box}
\[\mathop{\sum_{|t_1|\leq
T^{c/2}}}_{||t_2|-2|\leq T^{\epsilon-2}}\!\!1\ll \frac{T^{c/2-2+\epsilon}}{N(\fraka)^2}\quad \mbox{and }
\mathop{\sum_{|t_1|\leq T^{c/2}}}_{|t_2|<
2}\!\!1 \ll\frac{T^{c/2}}{N(\fraka)^2},\]
(note that we are not summing over $(t_1,t_2)=(2,2)$).
Plugging this back in (\ref{e:multbound2}) and recalling that $V(\fraka)\asymp N(\fraka)^3$ we get
\begin{eqnarray*}
 m(\pi,\Gamma(\fraka))T^{c|\frac{1}{2}-\frac{1}{p(\pi)}|}\ll_\epsilon  V(\fraka)
T^2+V(\fraka)^{1/3}T^{c/2-1+\epsilon}.
\end{eqnarray*}

\subsection{Proof of Corollary \ref{c:upperbound}}
Assume that $\fraka$ is prime to $\mathfrak{d}$ and that $\pi$ is a new representation occurring in $L^2(\Gamma(\fraka)\bs G^d)$ with $p(\pi)>6+\alpha$.
We combine the upper bound for $m(\pi,\Gamma(\fraka))$ given in Theorem \ref{t:multiplicity2} with the multiplicity lower bound (\ref{e:MLB}) and show that it now gives an upper bound for $T(\pi)$.

Let $\delta,\epsilon>0$ such that $\frac{\alpha}{2}>\delta>\epsilon$ and apply Theorem \ref{t:multiplicity2} with $c=6+\alpha-2\delta-2\epsilon$ to get
\[m(\pi,\Gamma(\fraka))\ll_\epsilon V(\fraka)^{1/3} T(\pi)^{\frac{\alpha+4}{\alpha+6}\delta+\epsilon-\frac{\alpha}{2}}(V(\fraka)^{2/3}+T(\pi)^{\frac{\alpha}{2}-\delta}).\]
Since $\pi$ occurs as a new representation we have $m(\pi,\Gamma(\fraka))\gg_\epsilon V(\fraka)^{1/3-\epsilon}$. Combining this with the above inequality we get that there is some constant $C(\epsilon)$ such that
\[1\leq C(\epsilon) V(\fraka)^{\epsilon} T(\pi)^{\frac{\alpha+4}{\alpha+6}\delta+\epsilon-\frac{\alpha}{2}}(V(\fraka)^{2/3}+T(\pi)^{\frac{\alpha}{2}-\delta}).\]

If $V(\fraka)^{2/3}\leq T(\pi)^{\frac{\alpha}{2}-\delta}$, we get that
\[1\leq 2C(\epsilon) T(\pi)^{\epsilon\frac{3(\alpha-2\delta)+4}{4}-\delta\frac{2}{\alpha+6}}.\]
Fixing $\epsilon_0=\epsilon_0(\alpha,\delta)$ sufficiently small we get that $T(\pi)^{\frac{\delta}{\alpha+6}}\leq 2C(\epsilon_0(\alpha,\delta))$, and hence, $T(\pi)$ is uniformly bounded. We thus get that in any case
$$T(\pi)\leq \max\{\tilde{C}(\alpha,\delta),V(\fraka)^{\frac{4}{3(\alpha-2\delta)}}\}\ll_\epsilon V(\fraka)^{\frac{4}{3\alpha}+\epsilon}.$$

%\subsection{Proof of Theorem \ref{t:main}}
%Assume that $\pi$ is a new representation occurring in $L^2(\Gamma(\fraka)\bs G^d)$ with $p(\pi)>7+\sqrt{17}+\delta$.
%Combining Corollaries \ref{c:lowerbound} and \ref{c:upperbound} we get
%\[V(\fraka)^{\frac{2\alpha}{3(8+\alpha)}-\epsilon}\ll_\epsilon T(\pi)\ll_\epsilon V(\fraka)^{\frac{4}{3\alpha}+\epsilon},\]
%with $\alpha=1+\sqrt{17}+\delta$. Consequently, there is some constant $C(\epsilon)$ such that
%\[V(\fraka)^{\frac{2\alpha^2-4\alpha-32}{3\alpha(8+\alpha)}-\epsilon}\leq C(\epsilon).\]
%For this choice of $\alpha$ we have that $\frac{2\alpha^2-4\alpha-32}{3\alpha(8+\alpha)}=c(\delta)>0$ and choosing $\epsilon_0=c(\delta)/2$ sufficiently small we get that $V(\fraka)$ and hence $N(\fraka)$ is bounded.

%\subsection{Proof of Theorem \ref{t:spectralgap}}
%Theorem \ref{t:spectralgap} is now a direct consequence. Let $N_0$ be sufficiently large such that when $N(\fraka)>N_0$, there are no new representations occurring in $L^2(\Gamma(\fraka)\bs G^d)$ with $p(\pi)\geq 12$.
%Let $C_0(\Gamma)=\max\{p(\Gamma(\mathfrak b))|N(\mathfrak b)\leq N_0\}$. Then for all $\fraka$ prime to $\mathfrak{d}$ we have that $p(\Gamma(\fraka))\leq \max\{C_0(\Gamma),12\}$. Indeed, if $\pi$ occurs in $L^2(\Gamma(\fraka)\bs G^d)$ there is some divisor $\mathfrak{b}|\fraka$ such that $\pi$ occurs in $L^2(\Gamma(\mathfrak b)\bs G^d)$ as a new representations. Then either $N(\mathfrak b)>N_0$ in which case $p(\pi)\leq 12$ or $N(\mathfrak b)\leq N_0$ and $p(\pi)\leq C_0(\Gamma)$.
%

\appendix
\section{Arithmetic multiplicity}\label{s:amult}
The proof of the bound (\ref{e:MLB}) for the multiplicities of new representations in $L^2(\Gamma(\frakp)\bs G^d)$ for $d\geq 2$ follows from the arguments in the proof of \cite[Theorem 3.2]{SarnakXue91} when $\frakp$ is a prime ideal, and there is a natural generalization of these arguments for a composite ideal $\fraka$. This was done in {\cite{BourgainGamburdSarnak06,BourgainGamburdSarnak08}} for square-free ideals and in \cite{BourgainGamburd08II} for prime powers.
For the sake of completeness we will include a short proof below.

For a fixed lattice $\Gamma\subseteq G^d$, by the Strong Approximation Theorem (see \cite{Weisfeiler84}), there is a square free ideal $\mathfrak{d}$, we call the discriminant of $\Gamma$, such that if $\fraka$ is prime to $\frakd$, then $\Gamma(\fraka)\bs \Gamma  \cong \PSL(2,\calO_L/\fraka)$. We then have
\begin{prop}
Any new representation $\pi$ occurring in $L^2(\Gamma(\fraka)\bs G^d)$ satisfies $m(\pi,\Gamma(\fraka))\gg_\epsilon N(\fraka_1)^{1-\epsilon}$ where $\fraka_1$ is the maximal ideal satisfying that $\fraka_1|\fraka$ and $(\frakd,\fraka_1)=1$.
\end{prop}
\begin{proof}
Given a representation $\pi=\pi_1\otimes\cdots\otimes\pi_d$, let $\lambda(\pi)=(\lambda_1,\ldots,\lambda_d)\in \bbR^d$, where the $\lambda_j$'s are the eigenvalues of the Casimir operator on the $\pi_j$'s, and let $m\in\bbZ^d$ such that $\pi_j=\frakD_{m_j}$ when $m_j\neq 0$ and $\pi_j$ is spherical when $m_j=0$. We thus have that $m(\pi,\Gamma(\fraka))=\dim V_\lambda(\Gamma(\fraka),m)$.

Since $\Gamma(\fraka)\vartriangleleft \Gamma$ is normal, the left action $\psi(g)\mapsto\psi(\gamma g)$ of $\Gamma(\fraka)\bs \Gamma$ preserves $L^2(\Gamma(\fraka)\bs G^d,m)$. This action commutes with the Casimir operator (in each coordinate) and hence defines an action on the eigenspace $V_\lambda(\Gamma(\fraka),m)$.
We can decompose
\begin{equation}\label{e:decomposition}V_\lambda(\Gamma(\fraka),m)=\bigoplus_{\rho} V_\rho\end{equation}
into invariant subspaces of this action where on each $V_\rho$ the action corresponds to an irreducible representation $\rho$ of $\Gamma(\fraka)\bs \Gamma$.
Notice that if a representation $\rho$ of $\Gamma(\fraka)\bs \Gamma$ factors through $\Gamma(\frakb)\bs\Gamma$ for some $\frakb|\fraka$ (that is $\rho(\gamma)=1$ for all $\gamma\in \Gamma(\frakb)$), then the space $V_\rho\subset V_\lambda(\Gamma(\frakb),m)$. Since we assumed that $\pi$ is a new representation then $V_\lambda(\Gamma(\frakb),m)=\{0\}$ for all $\frakb|\fraka$ and hence all representation appearing in (\ref{e:decomposition}) do not factor through such quotients.

Let $\fraka=\fraka_0\fraka_1$ with $\fraka_1$ the maximal divisor prime to $\frakd$, then
$$\Gamma(\fraka)\bs \Gamma\cong \calG_0\times \PSL(2,\calO_L/\fraka_1),$$
where $\calG_0$ is a subgroup of $\Gamma(\fraka_0)\bs \Gamma$.
Let $\fraka_1=\prod_{j=1}^\omega\frakp_j^{e_j}$ with $\frakp_j$ prime ideals; by the Chinese Reminder Theorem
$\PSL(2,\calO_L/\fraka_1)\cong \prod_{j=1}^\omega\PSL(2,\calO_L/\frakp_j^{e_j})$.
We can thus identify any irreducible representation $\rho$ of $\Gamma(\fraka)\bs \Gamma$ with a product
$\bigotimes_{j=0}^{\omega}\rho_j$ where $\rho_0$ is an irreducible representation of $\calG_0$ and for $j\geq 1$ each $\rho_j$ is an irreducible representation of $\PSL(2,\calO_L/\frakp_j^{e_j})$. The condition that $\rho$ does not factor through any quotients $\Gamma(\frakb)\bs\Gamma$ implies that $\rho_j$ does not factor through the quotient $\PSL(2,\calO_L/\frakp_j^{k})$ for any $0\leq k<e_j$.

The dimensions of the irreducible representations of $\PSL(2,\calO_L/\frakp^e)$ were analyzed in \cite[Theorem 7.4]{JsikinZapirain06} and satisfy, in particular, that their dimension is at least $\frac{N(\frakp)^k}{3}$ where $k\leq e$ is the smallest integer such $\rho$ factors through $\PSL(2,\calO_L/\frakp^k)$.
Hence, in our case we have that
$$\dim V_{\rho}\geq \prod_{j=1}^\omega \frac{N(\frakp_j^{e_j})}{3}= 3^{-\omega}N(\fraka')\gg_\epsilon N(\fraka_1)^{1-\epsilon}.$$
Since there is at least one irreducible component, we have $m(\pi,\Gamma(\fraka))\gg_\epsilon N(\fraka_1)^{1-\epsilon}$.
\end{proof}

\section{Non-spherical representations for $d=2$}\label{s:nonspherical}
When $d=2$, we give an alternative proof of the multiplicity bound for the non-spherical case.
The bound we prove here is not as good as Theorem \ref{t:multiplicity1} in terms of the dependence on $T(\pi)$. However, it has the interesting feature that it holds for all lattices and not just congruence covers.
To simplify the argument we assume that $\Gamma$ is torsion-free. We remark that we can always find a finite index torsion free subgroup so we do not lose any generality.
\begin{prop}\label{p:nonspherical2}
Let $\Gamma\subseteq G^2$ denote an irreducible torsion-free co-compact lattice and $\pi$ a non-spherical non-tempered representation of $G^2$. Then
\[m(\pi,\Gamma\bs G^2)\leq 8T(\pi)\big(\vol(\Gamma\bs G^2)\big)^{\frac{2}{p(\pi)}}.\]
\end{prop}
The proof of this result does not depend on lattice counting arguments but rather on the following positivity argument coming from the structure of the trace formula in this case:
\begin{lem}
Let $h(r)$ be an even holomorphic function positive on $\bbR\cup i\bbR$ with Fourier transform even, positive and compactly supported.
For $m\in \bbN$ let $\pi_{k,m}=\pi_{1/2+ir_k(m)}\otimes\frakD_m$ denote all representations (of this type) occurring in $L^2(\Gamma\bs G^2)$.
Then
\begin{eqnarray*}
\sum_{k}m(\pi_{k,m},\Gamma)h(r_k(m))\leq 2m\bigg(h(i/2)+\frac{\vol(\Gamma\bs G)}{8\pi^2}\int h(r)r\tanh(\pi r)dr\bigg).
\end{eqnarray*}
\end{lem}
\begin{proof}
We recall the hybrid trace formula (see \cite[Theorem 7']{kelmer10}),
\begin{eqnarray*}
\lefteqn{\sum_{k}m(\pi_k(m),\Gamma)h(r_k(m))-\delta_{m,1}h(i/2)=}\\
&&=\frac{(2m-1)\vol(\Gamma\bs G)}{(4\pi)^2}\int_\bbR h(r)r\tanh(\pi r)dr
-\frac{1}{2}{\sum_{\{\gamma\}}}'\tfrac{\vol(\Gamma_\gamma\bs G_\gamma^2)\hat{h}(l_\gamma)}{2\sinh(l_\gamma/2)}\tfrac{\sin((m-\frac{1}{2})\theta_\gamma)}{\sin(\theta_\gamma/2)}.
\end{eqnarray*}
where the notation $\sum'$ indicates that we are summing over conjugacy classes of elements satisfying $|\Tr(\gamma_2)|<2$.

In particular, when $m=1$ from the positivity of $\hat{h}$ we get
\begin{eqnarray*}
\sum_{k}m(\pi_{k,1},\Gamma)h(r_k(1))\leq h(i/2)+\frac{\vol(\Gamma\bs G)}{16\pi^2}\int_\bbR h(r)r\tanh(\pi r)dr,
\end{eqnarray*}
which proves the bound for $m=1$.
On the other hand, from the positivity of $h$ we also have
\begin{eqnarray*}
\tfrac{1}{2}{\sum_{\{\gamma\}}}'\frac{\vol(\Gamma_\gamma\bs G_\gamma)\hat{h}(l_\gamma)}{2\sinh(l_\gamma/2)}\leq h(i/2)+\frac{\vol(\Gamma\bs G)}{16\pi^2}\int h(r)r\tanh(\pi r)dr.
\end{eqnarray*}
Now for $m>1$, from the bound $\left|\frac{\sin((m-\frac 1 2)\theta)}{\sin(\theta/2)}\right|\leq 2m-1$ we get
\begin{eqnarray*}\lefteqn{\sum_{k}m(\pi_{k,m},\Gamma)h(r_k)}\\
&&\leq (2m-1)\bigg(\frac{\vol(\Gamma\bs G)}{16\pi^2}\int h(r)r\tanh(\pi r)dr+\tfrac{1}{2}{\sum_{\{\gamma\}}}'\frac{\vol(\Gamma_\gamma\bs G_\gamma)\hat{h}(l_\gamma)}{2\sinh(l_\gamma/2)}\bigg)\\
&&\leq 2m\bigg(\frac{\vol(\Gamma\bs G)}{8\pi^2}\int h(r)r\tanh(\pi r)dr+h(i/2)\bigg).
\end{eqnarray*}
\end{proof}

The proof of the multiplicity bound follows easily from this lemma.
\begin{proof}[Proof of Proposition \ref{p:nonspherical2}]
A non-tempered non-spherical representation is of the form $\pi=\pi_s\otimes\frakD_m$ with $s=\frac{1}{2}-\tau=\frac{1}{p(\pi)}$ and $0\neq m\in \bbZ$. In this case $|m|\leq T(\pi)\leq |m|+1$ and since $m(\pi_s\otimes\frakD_m,\Gamma)=m(\pi_s\otimes\frakD_{-m},\Gamma)$ we may assume $m>0$.

Let $\Psi$ denote a positive even function with Fourier transform even positive and compactly supported such that $\hat{\Psi}(0)=1$ and $1\leq \Psi(i\tau)\leq 2$ for $\tau\in [0,\tfrac{1}{2}]$.
For $R>0$ let $h_R(r)=\frac{\sin^2(Rr)}{r^2}\Psi(r)$, so that for $\tau\in(0,\frac{1}{2}]$ we have
$$\frac{e^{2R\tau}}{4\tau^2}\leq h_R(i\tau)\leq 2\frac{e^{2R\tau}}{4\tau^2}.$$
The above lemma then gives
\begin{eqnarray*}
m(\pi_{s}\otimes \frakD_m,\Gamma)h_R(i\tau)\leq 2m\left(h_R(\tfrac{i}{2})+\frac{\vol(\Gamma\bs G)}{8\pi^2}\int h_R(r)r\tanh(\pi r)dr\right).\end{eqnarray*}
We can bound the integral $\tfrac{1}{8\pi^2}\int h_R(r)r\tanh(\pi r)dr\leq \int_\bbR \Psi(r)dr=1$,
implying that
\[m(\pi_{s}\otimes \frakD_m,\Gamma)e^{2R\tau}\leq 4m(e^{R}+\vol(\Gamma\bs G)).\]
Taking $R$ so that $e^{R}=\vol(\Gamma\bs G)$ implies
\[m(\pi_{s}\otimes \frakD_m,\Gamma)\leq 8m\left(\vol(\Gamma\bs G)\right)^{1-2\tau}.\]
\end{proof}
\begin{rem}
This proof is possible due to the minus sign of $h(i/2)$ in the trace formula. This is a special feature of $d=2$ (or in general when there is an odd number of non-spherical representation in the product). Consequently, this proof does not seem to generalize for $d>2$.
\end{rem}

%----------------------------------------------------------------
%%GATHER{C:/Bib/Mybib.bib}   % For Gather Purpose Only
%\bibliographystyle{amsalpha}
%\bibliography{C:/Bib/Mybib}

\begin{thebibliography}{Gam02}

\bibitem[BlBr]{BlomerBrumley10}
V.~Blomer and F.~Brumley, \emph{On the Ramanujan conjecture over number fields.}
   Preprint, \texttt{http://arxiv.org/abs/1003.0559/}.

\bibitem[BG]{BourgainGamburd08II}
J.~ Bourgain and A.~ Gamburd, \emph{Expansion and random walks in {${\rm SL}_d(\Bbb Z/p^n\Bbb Z)$} {I}.},
J. Eur. Math. Soc. \textbf{10} (2008), no.~4, 987--1011.
   \MR{2443926 (2010a:05093)}


\bibitem[BGS1]{BourgainGamburdSarnak06}
J.~ Bourgain, A.~ Gamburd, and P.~ Sarnak, \emph{Sieving and expanders.} C.
  R. Math. Acad. Sci. Paris \textbf{343} (2006), no.~3, 155--159. \MR{2246331
  (2007b:11139)}

\bibitem[BGS2]{BourgainGamburdSarnak08}
\bysame, \emph{Affine linear
sieve, expanders and sum-product.}, Preprint (2008), available at
\texttt{http://www.math.princeton.edu/sarnak/}

\bibitem[BGS3]{BourgainGamburdSarnak10}
\bysame, \emph{Generalization of selberg's 3/16 theorem and affine sieve.}
  Invent. Math. Preprint, \texttt{http://arxiv.org/abs/0912.5021}.

\bibitem[Ef]{Efrat87}
I.~Efrat, \emph{The selberg trace formula for $\mathrm{PSL}_2(\mathbb{R})^n$.}
  Mem. Amer. Math. Soc. \textbf{65}.

\bibitem[Ga]{Gamburd02}
A.~ Gamburd, \emph{On the spectral gap for infinite index ``congruence''
  subgroups of  {${\rm SL}\sb 2(\bold Z)$}.} Israel J. Math. \textbf{127}
  (2002), 157--200.

\bibitem[He]{Hejhal76}
D.~A. Hejhal, \emph{The {S}elberg trace formula for {${\rm PSL}(2,R)$.}
  {V}ol. {I}}, Springer-Verlag, Berlin, 1976, Lecture Notes in Mathematics,
  Vol. 548.

\bibitem[Iw]{Iwaniec02}
H.~Iwaniec, \emph{Spectral methods of automorphic forms.}
American Mathematical Society, Providence, RI, 2002, Graduate Studies in Mathematics, Vol. 53.


\bibitem[JaL]{JacquetLanglands70}
H.~Jacquet and R.~P. Langlands, \emph{Automorphic forms on {${\rm GL}(2)$}.}
  Springer-Verlag, Berlin, 1970, Lecture Notes in Mathematics, Vol. 114.


\bibitem[Jai]{JsikinZapirain06}
A.~Jaikin-Zapirain, \emph{Zeta function of representations of compact $p$-adic analytic groups.} J. Amer. Math. Soc. \textbf{19} (2006), no.~1, 91--118 \MR{2169043 (2006f:20029)}


\bibitem[Ke]{kelmer10}
D.~Kelmer, \emph{Distribution of holonomy about closed geodesics in a product
  of hyperbolic planes.}, Preprint \texttt{http://arxiv.org/abs/0911.0329}

\bibitem[KeSa]{kelmerSarnak09}
D.~Kelmer and P.~Sarnak, \emph{Strong spectral gaps for compact quotients
  of products of $\PSL(2,\bbR)$.} J. Eur. Math. Soc. \textbf{11} (2009), no.~2,
  283--313.


\bibitem[Ko]{Kontorovich09}
A.~V. Kontorovich, \emph{The hyperbolic lattice point count in infinite
  volume with applications to sieves.} Duke Math. J. \textbf{149} (2009),
  no.~1, 1--36. \MR{2541126}

\bibitem[KiSa]{KimSarnak03}
H.~H. Kim and P.~Sarnak, \emph{Refined estimates towards the ramanujan
  and selberg conjectures.} J. Amer. Math. Soc. \textbf{16} (2003), no.~1,
  139--183 (electronic), Appendix to H.~H. Kim, \emph{Functoriality for the exterior square of {${\rm GL}\sb 4$}
  and the symmetric fourth of {${\rm GL}\sb 2$}}.


%\bibitem[KiSh]{KimShahidi02}
%H.~H. Kim and F.~Shahidi, \emph{Functorial products for {${\rm GL}\sb
%  2\times{\rm GL}\sb 3$} and the symmetric cube for {${\rm GL}\sb 2$}.} Ann. of
%  Math. (2) \textbf{155} (2002), no.~3, 837--893, With an appendix by Colin J.
%  Bushnell and Guy Henniart.


\bibitem[LuSe]{LubotzkyAlexander03}
A.~Lubotzky and D.~Segal, \emph{Subgroup growth.} Progress in
  Mathematics, vol. 212, Birkh\"auser Verlag, Basel, 2003.

\bibitem[Ma]{Margulis91}
G.~A. Margulis, \emph{Discrete subgroups of semisimple {L}ie groups.}
  Ergebnisse der Mathematik und ihrer Grenzgebiete (3) [Results in Mathematics
  and Related Areas (3)], vol.~17, Springer-Verlag, Berlin, 1991. \MR{1090825
  (92h:22021)}

\bibitem[SaXu]{SarnakXue91}
P.~Sarnak and X.~X. Xue, \emph{Bounds for multiplicities of automorphic
  representations.} Duke Math. J. \textbf{64} (1991), no.~1, 207--227.


\bibitem[Se1]{Selberg65}
A.~Selberg, \emph{On the estimation of {F}ourier coefficients of modular
  forms.} Proc. Sympos. Pure Math., Vol. VIII, Amer. Math. Soc., Providence,
  R.I., 1965, pp.~1--15.

\bibitem[Se2]{Selberg95}
\bysame, \emph{Partial zeta function.} Mittag-Leffler Inst. lecture notes,
  1995.


\bibitem[Weil]{Weil60}
A.~Weil, \emph{Algebras with involutions and the classical groups.} J.
  Indian Math. Soc. (N.S.) \textbf{24} (1960), 589--623 (1961).

\bibitem[Weis]{Weisfeiler84}
B.~Weisfeiler, \emph{Strong approximation for {Z}ariski-dense subgroups of
  semisimple algebraic groups.} Ann. of Math. (2) \textbf{120} (1984), no.~2,
  271--315. \MR{763908}
\end{thebibliography}

%*********************************************************************
%*********************************************************************
\def\cprime{$'$} \def\cprime{$'$}
\providecommand{\bysame}{\leavevmode\hbox to3em{\hrulefill}\thinspace}
\providecommand{\MR}{\relax\ifhmode\unskip\space\fi MR }
\providecommand{\MRhref}[2]{%
  \href{http://www.ams.org/mathscinet-getitem?mr=#1}{#2}
}
\providecommand{\href}[2]{#2}

%%*********************************************************************

\end{document}